\documentclass{amsart}

\usepackage{amssymb,amsmath}
\usepackage{url}
\usepackage{amscd}
\usepackage{texdraw}
\usepackage[all]{xy}
\usepackage{fancyhdr}
\usepackage{amsmath}
\usepackage{graphicx}
\usepackage{subcaption}

\usepackage{amscd,color}
\usepackage[shortlabels]{enumitem}

\theoremstyle{plain}
\newtheorem{thm}{Theorem}[section]
\newtheorem*{starthm}{Theorem}

\newtheorem{cor}[thm]{Corollary}
\newtheorem{prop}[thm]{Proposition}
\newtheorem{lemma}[thm]{Lemma}
\newtheorem*{starlemma}{Lemma}
\newtheorem{defn}{Definition}

\newtheorem{remark}{Remark}[section]

\input txdtools

\newcommand{\calb}{{\mathcal B}}

\newcommand{\calf}{{\mathcal F}}
\newcommand{\calfp}{{\mathcal {FP}}}

\newcommand{\calh}{{\mathcal H}}

\newcommand{\call}{{\mathcal L}}
\newcommand{\calm}{{\mathcal M}}

\newcommand{\calr}{{\mathcal R}}

\newcommand{\CC}{{\mathbb C}}
\newcommand{\DD}{{\mathbb D}}

\newcommand{\NN}{{\mathbb N}}

\newcommand{\RR}{{\mathbb R}}

\newcommand{\ZZ}{{\mathbb Z}}

\renewcommand{\hat}{\widehat}

\newcommand{\la}{\lambda}

\begin{document}
\title{Meromorphic functions with a polar asymptotic value}

\author{Tao Chen and Linda Keen}

\address{Tao Chen, Department of Mathematics, Engineering and Computer Science,
Laguardia Community College, CUNY,
31-10 Thomson Ave. Long Island City, NY 11101 and
CUNY Graduate Center, New York, NY 10016
}
\email{tchen@lagcc.cuny.edu}

\address{Linda Keen, Department of Mathematics, the CUNY Graduate
School, New York, NY 10016}
\email{LKeen@gc.cuny.edu; linda.keenbrezin@gmail.com}

\thanks{This material is based upon work supported by the National Science Foundation under Grant No. 1440140, while the second author was in residence at the Mathematical Sciences Research Institute in Berkeley, California, during the spring semester 2022}


\begin{abstract}
This paper is part of a general program in complex dynamics to understand parameter spaces of transcendental maps with finitely many singular values.

 The simplest families of such functions have two asymptotic values and no critical values.  These families, up to affine conjugation, depend on two complex parameters. Understanding their parameter spaces is key to understanding families with more asymptotic values, just as understanding quadratic polynomials was for rational maps more generally.

The first such families studied were the one-dimensional slices of the exponential family, $\exp(z) + a$, and the tangent family $\lambda \tan z$.  The exponential case exhibited phenomena not seen for rational maps: Cantor bouquets in both the dynamic and parameter spaces, and no bounded hyperbolic components.  The tangent case, with its two finite asymptotic values $\pm \lambda i$, is closer to the rational case, a kind of infinite degree version of the latter.

In this paper, we consider a general family that interpolates between $\exp(z) + a$ and $\lambda \tan z$.  Our new family has two asymptotic values and a one-dimensional slice for which one of the asymptotic values is constrained to be a pole, the ``polar asymptotic value'' of the title.   We show how the dynamic and parameter planes for this slice exhibit behavior that is a surprisingly delicate interplay between that of the $\exp(z) + a$ and $\lambda \tan z$ families.

\end{abstract}

\maketitle
\section{Introduction}

A general principle in complex dynamics is that a function's singular values,  the points over which it is not a local homeomorphism,``control'' the eventually periodic stable dynamics.   Since rational maps have only finitely many singular values, and all stable behavior is eventually periodic,  (see \cite{Sul}), a given family exhibits only finitely many different stable dynamical phenomena.  The same is true for families of transcendental functions with finitely many singular values, (see \cite{BKL4}), but in this case some of these values must be asymptotic, in the sense that, if  $f$ is the map, there is a path $\gamma(t)$ such that $\lim_{t \to 1} \gamma(t)=\infty$, then  $\lim_{t \to 1} f(\gamma(t))=v$.  Where transcendental functions and rational functions part company more seriously is due to transcendentals having infinite degree, and having a point, infinity, that is an essential singularity and not in the domain.  This paper is part of a program to understand both the dynamical properties of functions in such families and the structure of their parameter spaces.

The artistry here is in describing families that are natural and have something interesting to say.   The obvious starting point  (see for example,  \cite{D1, DG,DT, S} was the exponential family $E_a(z)=\exp(z) +a$. Here infinity not only has no forward image, it also has no preimages.  It is, however, an asymptotic value of $E_a$.  The only other singular value is the asymptotic value $a$ and as $a$ varies, so do the dynamics.  For example, the unstable of $E_a$ set contains ``Cantor bouquets'' of curves, which are families of  curves invariant under $E_a$ whose cross-section is a Cantor set.  In addition, all components of the parameter plane for the family $E_a$ in which the functions have topologically conjugate dynamics are unbounded.

One natural family, then, generalizes $E_a(z)$ to transcendental meromorphic functions with no critical values and exactly two asymptotic or omitted values. In addition to the exponential, this family  contains the tangent family, $\la \tan z$;  its asymptotic values are $\pm \la i$ and their orbits are symmetric.  Thus, the dynamics are ``controlled'' by the behavior of either of the asymptotic values.   There are parameter values for which the asymptotic values of the functions are prepoles, and these play a special role in the structure of the parameter plane.  See \cite{ ABF,CJK1,FK, KK}.

 Another variation on this theme is to consider the  family of functions with two finite asymptotic values and one attracting fixed point.  Fixing the derivative at the attracting fixed point determines a one dimensional parameter space.  There are tangent maps  contained in this family and the local dynamical structure is similar to what we see for $\la \tan z$.  Because the asymptotic values are not symmetric, however, the parameter plane has properties similar to both the tangent family and rational maps of degree $2$ with an attracting fixed point.    See \cite{CJK2,CJK3}.

In this paper, we are going to look at a subfamily $\calfp_2$ of the family with two asymptotic values and no other singular values.  The functions in $\calfp_2$,  have the property that one asymptotic value is fixed and is a pole --- the ``polar asymptotic value''; the other asymptotic value is ``free'' and determines the dynamics.    The family $\calfp_2$ can be thought of as a hybrid of the tangent and exponential families: in both the dynamical and parameter planes, we have behavior that reflects properties from each of these two families.  The significance of having a polar asymptotic value stems from the general fact that  an attracting periodic cycle must attract an asymptotic value, but it cannot attract the pole, so the only other asymptotic value must be attracted to it.   This implies there can be only one attractive cycle, and its basin of attraction is the full stable set.

On the one hand, this behavior is similar to the exponential family, where one asymptotic value is infinity (a pole of order 0). On the other hand, because here there are prepoles, the non-zero asymptotic value can be a prepole, so we also see similarities with the tangent family.   It is this interplay that raises new problems and makes this family of particular interest.

Our analysis of $\calfp_2$ gives an explicit description of ideas introduced in a rather different context, \cite{FK} and extended in \cite{ABF}, in which the families studied have a finite number of singular values and are parameterized by these values.   
In these families, ``slices'' are defined by fixing  the dynamic behavior of the  orbits of some of  singular values and allowing the dynamic behavior of the others to vary freely.  There are components of this slice in which the functions have topologically conjugate dynamics.  In these components, the functions are hyperbolic 
 meaning that the orbits of the free singular values  are disjoint from the unstable set,   the Julia set.     This property allows these components to be fully characterized, and in particular, it was shown that they have a unique special boundary point that can be characterized in two ways:   first, the free asymptotic value is a prepole and second, as this boundary point is approached from within the component, the derivative along the attracting cycle tends to zero.   We called points with these properties {\em virtual cycle parameters} and {\em virtual centers}, respectively.

To return to our family $\calfp_2$:  The polar asymptotic value always belongs to the unstable set, so the functions in $\calfp_2$ can not be hyperbolic.  Nevertheless, the parameter space contains components in which the dynamics are topologically conjugate.  As shown in \cite{ABF}, these are ``stability components'' in the sense of Ma\~n\'e-Sad-Sullivan, \cite{MSS}.   We will describe how these ``pseudo-hyperbolic components'' are situated as a function of the free asymptotic value.  There are values of the parameter for which the free asymptotic value is a prepole, and this is where  ideas from  \cite{FK}  come in. These values are virtual cycle parameters, and they again play a role  as virtual centers, in this case, for the pseudo-hyperbolic components.

Before getting to the general theorems, let us look at the somewhat special case where the attractive cycle is a fixed point, whose dynamical picture is very much like that of $E_a(z)$, where $a$ is attracted to a fixed point.  By analogy to that case, using the definition of Cantor bouquets in \cite{DK, DT} and the techniques of \cite{DK},  we prove:

\begin{starthm}[Theorem A]
Let $f \in \calfp_2$ and suppose the free asymptotic value is also a prepole.  Then 
the Julia set of $f$ contains  Cantor bouquets at infinity as well as at each of the poles and prepoles;   if, in addition,  $f$ has an attracting fixed point,  the stable set is a completely invariant connected set and the Cantor bouquets comprise its full complement.  \end{starthm}
\medskip

 Going back to the early work on quadratic maps (see  e.g. \cite{DH}), a standard approach to studying both the dynamic and parameter spaces is to find a   scheme to encode the dynamics.  Here we define a combinatorial tree-like structure whose nodes are prepoles and whose branches are preimages of segments of the imaginary axis.  The labels, or addresses, of the nodes are determined by the order of the poles on the imaginary axis.   In addition, we study the dynamics by defining a structure on the set of  preimages of the asymptotic tracts, $L$ and $R$, which we call the L-R structure.
   
 Kneading sequences for attractive cycles of functions in various families of holomorphic functions have been defined in the literature  by going backwards around the cycle and reading from right to left, (see e.g. \cite{CJK2,DFJ,KK}).  Here,  we also define a kneading sequence for each attractive cycle using the addresses of the prepoles and  the L-R structure.   We write the kneading sequence as $*k_1\cdots k_{n-1}$ where the $*$ represents an ambiguity we discuss later.

 The polar asymptotic value makes  the ``L-R structure'' particularly interesting when the function has an attractive cycle.  
  A rough description is that points in the attractive cycle can bounce back and forth from a neighborhood of infinity to a neighborhood of the polar asymptotic value.   If that doesn't happen, we call the cycle {\em regular}.  If part of the cycle bounces, we call the cycle {\em hybrid}, and if the whole cycle bounces, we call it {\em unipolar}.   See Definitions~\ref{knseqnew} and \ref{poleseq} for the formal definitions.

 We can completely characterize the types of kneading sequences that can occur:
\begin{starthm}[Theorem B]
 Let $f$ be a pseudo-hyperbolic map whose attracting cycle $\{z_0, \ldots, z_{n-1}\}$ has period $n$ and suppose the cycle has kneading sequence $*k_1\cdots k_{n-1}$.

Then the  kneading sequence for $f$ is one of:
\begin{enumerate}
\item  Regular sequence: $*k_1\cdots k_{n-1}$ where $k_{n-1}\neq 0$ or
 \item Unipolar sequence:  $*k_10 \cdots k_{n-2}0$,  the period is odd so the sequence contains an even number of digits after the $*$;    all the even ones  but not all the odd ones are zero,
\item Hybrid sequence:  $*k_1\cdots k_{l}k_{l+1}0\cdots k_{n-2}0$, $ l >1$, $k_l, k_{l+1} \neq 0$.
\end{enumerate}
\end{starthm}
If $n>1$, not all entries in the kneading sequence can be zero and  if $n$ is even, the sequence cannot have the form $*0k_20 \ldots 0k_{n-1}0$.

Functions in the same pseudo-hyperbolic component have the same kneading sequences, so the kneading sequences can be used to label these components.
Because these functions have no critical values, the derivative at any point in the cycle can never be zero.   A boundary point at which this derivative does have limit $0$ is again called a {\em  virtual center.}   Using techniques similar to those in \cite{FK} we prove

\begin{starthm}[Theorem C]   Every pseudo-hyperbolic  component has a unique virtual center and this boundary point is a virtual cycle parameter.
\end{starthm}

The main result is essentially a converse to  Theorems B and C.

\begin{starthm} [Theorem D]
\begin{enumerate}
\item Given any unipolar kneading sequence, there is a  pseudo-hyperbolic component with this kneading sequence.  It is unbounded in the  half-plane, and as the parameter tends to $\infty$ in the component, the multiplier tends to zero.
\item Given any regular kneading sequence, there is a   component with this sequence and its virtual center is the virtual cycle parameter corresponding to the sequence.
\item Given any hybrid kneading sequence, there is a  component with this sequence;  its virtual center is the virtual cycle parameter corresponding to the first $l$ terms of the sequence .
\end{enumerate}
\end{starthm}

The paper is organized as follows:

We start by briefly reviewing the basic theory.  The family $\calfp_2$  and its virtual cycles are defined, and then the rest of the first half of the paper is devoted to proving Theorems A and B.  A key first step to understanding the dynamics is to examine the effect of the polar asymptotic value on the L-R structure of the pre-asymptotic tracts of both asymptotic values and constructing the combinatorial tree.  This leads to a proof of theorem A.  We can then define kneading sequences for  functions in $\calfp_2$, which leads into the proof of theorem B.

In the second half of the paper, we turn to the parameter space and prove theorems C and D.  We conclude with some results about parameters in the complement of the pseudo-hyperbolic components and describe some interesting open questions.

The authors would like to thank the referee of an earlier version of this paper whose careful reading and comments have improved the exposition.    

\section{Dynamics}
\label{Basic Dynamics}
 Here we review the basic definitions, concepts and notations we will use.  We refer the reader to standard sources on meromorphic dynamics for details and proofs.  See e.g. \cite{B, BF,BKL1,BKL2,BKL3,BKL4,DK,KK}.

 We denote the complex plane by $\CC$, the Riemann sphere by $\hat\CC$ and the unit disk by $\DD$.  We denote the punctured plane by $\CC^* = \CC \setminus \{  0 \}$ and the punctured disk by $\DD^* = \DD \setminus \{  0 \}$.

\medskip
Given  a family of meromorphic functions, $\{ f_{\la}(z) \}$,  we look at the orbits of points formed by iterating the function $f(z)=f_{\la}(z)$.   If $f^k(z)=\infty$ for some $k>0$, $z$ is called a prepole of order $k$ --- a pole is a prepole of order $1$.  For meromorphic functions, the poles and prepoles have  finite orbits that end at infinity.  The {\em Fatou set or stable set, $F_f$}, consists of those points at which the iterates form a normal family.   The Julia set $J_f$ is the complement of the Fatou set and contains all the poles and prepoles.

\medskip
A point $z$ such that $f^n(z)=z$ is called {\em periodic}.  The minimal such $n>0$ is called the {\em minimal period}. Periodic points are classified by their multipliers, $m(z)=(f^n)'(z)$ where $n$ is the minimal period: they are repelling if $|m(z)|>1$, attracting if $0< |m(z)| < 1$,   super-attracting  if $m=0$ and neutral otherwise.  A neutral periodic point is {\em parabolic} if $m(z)=e^{2\pi i p/q}$ for some rational $p/q$.  The Julia set is the closure of the repelling periodic points.  For meromorphic $f$, it is also the closure of the prepoles, (see e.g. \cite{BKL1}).

\medskip
If $D$ is a component of the Fatou set,  either $f^n(D) \subseteq f^m(D)$ for some integers $n,m$ or $f^n(D) \cap f^m(D) = \emptyset$ for all pairs of integers $m \neq n$.  In the first case $D$ is called {\em eventually periodic} and in the latter case it is called {\em wandering}.   The   periodic cycles of stable domains are classified as follows:
\begin{itemize}
\item Attracting or super attracting if the periodic cycle of domains contains an attracting or superattracting cycle in its interior.
\item Parabolic if the iterates of points in $D$ converge to a parabolic periodic point on its boundary.
\item Rotation if $f^n: D \rightarrow D$ is holomorphically conjugate to a rotation map.  Rotation domains are either simply connected or topological annuli.  These are called {\em Siegel disks and Herman rings} respectively.
\item Essentially parabolic, or Baker, if there is a point $z_{\infty} \in \partial D$ such that for every $z \in D$,  $\lim_{k \to \infty} f^{nk}(z) = z_{\infty}$ and $f$ is not defined at the  point  $z_{\infty}$.
\end{itemize}

\medskip
A point $v$ is a {\em singular value} of $f$ if $f$ is not a regular covering map over $v$.
\begin{itemize}
\item    $v$ is a {\em critical value} if for some $z$, $f'(z)=0$ and $f(z)=v$.
\item    $v$ is an {\em asymptotic value} for $f$ if there is a path $\gamma(t)$ such that \\
$\lim_{t \to \infty} \gamma(t) = \infty$ and $\lim_{t \to \infty} f(\gamma(t))=v$.
\item  $v$ is a {\em logarithmic asymptotic value} if $U$ is a neighborhood of $v$ and there is some component $V$ of $f^{-1}(U \setminus \{v\})$, called an {\em asymptotic tract of $v$}, on which $f$ is a universal covering.\footnote{We identify asymptotic tracts if the asymptotic paths they contain are homotopic in their intersection.}. Isolated asymptotic values are always logarithmic.
\item The {\em set of singular values $S_f$} consists of the closure of the critical values and the asymptotic values.  The {\em post-singular set is
\[P_f= \overline{\cup_{v \in S_f} \cup_{n=0}^\infty f^n(v)}. \]}
For notational simplicity, if a prepole $p$ of order $n$ is a singular value, $\cup_{k=0}^{n} f^k(p)$ is a finite set with $f^{n} (p)=\infty$.
\end{itemize}

\medskip
A standard result in dynamics is that each attracting, superattracting, or parabolic cycle of domains contains a singular value. Moreover, unless the cycle is superattracting, the orbit of the singular value is infinite and accumulates on the cycle.  The boundary of each rotation domain is contained in the accumulation set of the forward orbit of a singular value.  (See e.g.~\cite{M}, chap 8-11 or~\cite{B}, Sect.4.3.)

\medskip

The  meromorphic  functions in the family  $\calfp_2$, which are the focus of this paper, have no critical values and exactly two finite simple asymptotic values, one of which is always pole.    Since the derivative of such a function is never zero, no function in $\calfp_2$ can have a superattracting cycle.

It follows from \cite{DK,BKL4} that functions in  $\calfp_2$   never have wandering domains and from \cite{DK} that they do not have Baker domains.
The arguments in \cite{CK2} that prove  that there are no Herman rings for the family $\la\tan^p z^q$   also show the same is true for the family $\calfp_2$.
 \medskip

\section{The family $\calfp_2$}
\label{family}
Functions with exactly two asymptotic values and no critical values are M\"obius transformations of the exponential function $e^{kz}$ for arbitrary $k\in \mathbb{C}^*$.   Because the dynamics are invariant under conjugation by a conformal isomorphism, we may fix the essential singularity at infinity and set $k=2$.   We will restrict to those functions for which one asymptotic value is a pole, and we may take that pole to be $0$.  If the other asymptotic value is denoted by $\la \in \CC^*$, the function can be written as
$$ f_{\la}(z)=f(z) = \frac{\lambda}{1-e^{-2z}}. $$
Set $\calfp_2=\{ f_{\la}(z), \la \in \CC^* \}$.
 Since the asymptotic value $0$ is a pole, the dynamics of $f_\lambda$ are determined only by the orbit of $\lambda$.   We call this the family of {\em unipolar maps}.

\subsection{Combinatorics of the prepoles}
Note that $$f_\lambda: \mathbb{C}\to \widehat{\mathbb{C}}\setminus \{0,\lambda\}$$ is a universal covering and therefore there are infinitely many inverse branches defined in a neighborhood of  every point except the two asymptotic values.   These preimages are disjoint if the neighborhoods are small enough and the local map is a homeomorphism.   Punctured neighborhoods of the asymptotic values have a simply connected preimage, the asymptotic tract, that is a universal covering map onto the punctured neighborhood.    The asymptotic tracts for $0$ and $\la$ are respectively, for some $M>0$, half planes $L_M=\{z | \Re z < -M \}$ and $R_M=\{z | \Re z > M \}$.  

Although the inverse branches are defined locally, they can be continued analytically.  Because  it is difficult to keep track of these individual branches as $\la$ varies, we use the poles and prepoles to obtain a combinatorial tree-like structure to understand both the dynamics and parameter space for these functions.  

 Assume first that $\la \not\in \Im=\{z=is, s \in \RR\}$ and that $\la$ is not a prepole. 
The poles of every $f_{\la}$ are the same:  $p_k= k\pi i$, $k \in \ZZ$.   There are  curves $C^{\pm}_k$ with a common endpoint at each pole $k \pi i$ that are connected preimages\footnote{A small neighborhood of each pole maps to a neighborhood of infinity and the inverse of this map takes the imaginary axis to curves ending at the pole.  The curves $C^{\pm}_k$ are then defined by analytic continuation.}    of the positive and negative imaginary axes, $\Im^{\pm}$.   They have a common tangent at $k\pi i$ in a direction perpendicular to $\arg \la$.  Set $C_k = C_k^- \cup C_k^+ \cup \{k \pi i\}$.  The $C_k$ are disjoint and $C_{k+1}=C_k + \pi i$.    

Each $C_k$ contains  one preimage of each pole, a prepole of order $2$.  
Because the non-zero poles are preimages of $\infty$,   they have a   natural the ordering  coming from the imaginary axis;    label them   by $p_{kj}$ where $f(p_{kj})= p_j$.  Since the poles accumulate to $\infty$,   $\lim_{j \to \pm\infty} p_{kj} = p_k$.   The segments of $C_k$ from the prepoles $p_{k1}$ and $p_{k -1}$  to $\infty$ contain no prepoles so $p_{k0}$ is not defined.  

At each prepole $p_{kj}$, there is a curve that is a connected preimage of $C_{kj}$; it contains prepoles of order $2$ labeled so that $p_{kji}=p_{ji}$,   As above, the prepoles of order $3$ accumumlate on prepoles of order $2$, and the prepoles $p_{kj0}$ are not defined.  
Continuing inductively, we assign a unique {\em address}  $k_1 \ldots k_n$, $k_j \in \ZZ, k_n \neq 0$, to each prepole of order $n$.  

 Since $f(p_{k_1 \ldots k_nk_{n+1}})=p_{k_1 \ldots k_n}$, $f$ induces a shift map on the space of finite sequences of  integers.   
\medskip

Next, suppose $\la = it$,  for some $t \in \RR$, not a multiple of $\pi$. Assume first that $t>0$. 
    Note that if $x \neq 0 \in \RR$,  $f(x) = it/(1-e^{-2x}) \in \Im$.   As $x \to 0$, $f(x) \to \infty$, and as above we can find preimages $p_{0k} \in \RR$, $|k|>>0$  of the poles $p_k$.   Since $\lim_{x \to -\infty} f(x) = 0$,  
     $f(\RR^-)=\Im^-$ and all the negative poles have preimages in  $\RR^-$; they are labeled $p_{0j}$, $j=-1,-2, \ldots$.  Because, however, $\lim_{x \to +\infty}f(x) = it$,  $\RR^+$ does not constitute a full preimage of $\Im^+$.   While the full  analytic continuation of the preimage of $\Im^+$  is $\RR^+$,  if $t> m\pi$ for some $m>0$,  it does not contain the preimages of the poles between $0$ and $it$.   The prepoles on $\RR^+$ are thus labeled $p_{0j}$, $j  =m+1, m+2, \ldots $.
     
    The preimages of the missing interval, $\Lambda=(0, it)$, are the horizontal lines $\ell_n^+=x+ i(2n+1)\pi/2$, and if $t > m\pi$, each of these lines contains preimages of the poles $p_1, \ldots, p_m$.    We make the convention that   the prepoles on the line $\ell_0^+=x+i\pi/2$ are labeled by $p_{0j}$, $j=1, \ldots, m$.

If $t<0$, the roles of $\RR^{\pm}$ are interchanged and the prepoles are labeled accordingly;  in this case we relabel the horizontal lines as $\ell_n^-=x+ i(2n-1)\pi/2$, and make the convention that  the prepoles on the line $\ell_{0}^-=x-i\pi/2$ are labeled by $p_{0 -j}$, $j=1, \ldots, m$.     

In either case,  we set $C_0= \RR^- \cup \RR^+ \cup \ell  \cup \{0\}$ where  $\ell=\ell^{\pm}$ as $t$ is greater or less than $0$, so that $C_0$ contains a single preimage $p_{0j}$  of each pole. 

Next, for each  $ k \in \NN$, we   set $C_k=C_0 + k\pi i$,\footnote{By abuse of notation, we continue to call $C_k$ a ``curve'', although  it may not be connected.}.     The $C_k$ are disjoint  and each contains a single preimage $p_{kj}$ of each $p_k$.  
Continuing inductively, we define curves $C_{k_1 \ldots k_n}$ and assign a unique {\em address} $k_1 \ldots k_n$ to each prepole of order $n+1$ in   $C_{k_1 \ldots k_n}$.  The map $f$ induces a right shift map on addresses.  

If $\la=i m \pi$    it has no preimages. The discussion and addressing scheme above works the same way;  however,  there will be no poles whose rightmost entry is $m$.

\subsection{L-R Structure of the pre-asymptotic tracts in the dynamic plane}

In this section we see how the asymptotic tracts of $0$ and $\la$ and their preimages are deployed in the dynamic plane of  $f=f_{\la}$.  We have normalized the functions so that the asymptotic tract of $0$ is always the left half-plane $L$ and that of $\la$ is the a right half-plane $R$.   In order to put a structure on the preimages of these tracts, it will be useful to look at subsets of these half-planes bounded by vertical lines.  By abuse of notation these are also called the asymptotic tracts.    How the preimages of these half-planes are arranged changes with $\la$.   We call the preimages, or pullbacks, of
$L$ and $R$ obtained by iteration using all the branches of $f_{\la}^{-1}$, the {\em  L-R structure of $f_{\la}$}. The L-R structure is not only useful in understanding the dynamic plane, but as we shall see later, in understanding the parameter plane.  The case $\Im \la =0$ is special and its dynamics will be treated separately. 

 Given $M>0$, set $R_M= \{ z\  |\ \Re{z} >M \}$ and $L_M= \{ z \ |\ \Re{z} <- M \}$;  these are asymptotic tracts for $\la$ and $0$ respectively.  By our convention on notation, if
$M_1>M$,  $R_{M_1}, L_{M_1}$ are also asymptotic tracts of $\la$ and $0$ respectively.
  Unless it causes confusion we will write $L=L_M, R=R_M$.  
    We will use the addresses of the   prepoles   to label the preimages of the tracts.

There are preimages of each of the asymptotic tracts at each of the poles $p_k$;  label  these respectively $L_k=f^{-1}(L)$ and $R_k=f^{-1}(R)$.    Because  $L$ and $R$   have a common boundary point at infinity where they are parallel to the imaginary axis, the  $L_k$ and $R_k$ meet at the poles where their common tangent  is in the direction of  $C_k$ and is perpendicular to the interval $\Lambda$ joining the asymptotic values.   Orient the boundaries of  half planes $L$ and $R$ so that up is positive.   Pulling back, this induces an orientation on the 
   boundaries of $L_k$ and $R_k$.
  
As long as $k_1,k_2 \neq 0$, the inverse maps at the poles are local homeomorphisms.  We label the preimages of $L_{k_1}$ and $R_{k_1}$ at the pole $p_{k_1k_2}$ as $L_{k_1k_2}$ and $R_{k_1k_2}$ respectively and continue inductively to obtain preimages $L_{k_1k_2 \ldots k_n}$ and $R_{k_1k_2 \ldots k_n}$ at the prepole $p_{k_1k_2 \ldots k_n}$.

 Because $0$ is both an asymptotic value and a pole,  it plays a special role.  Since $f$ maps $L$ as a universal cover to a punctured neighborhood   $V$ of $0$ and $V$ has non-empty intersections with $L_0$ and $R_0$, there are infinitely many preimages of both $L_0$ and $R_0$ in $L$.  In order to label them, 
recall that $\RR^{\pm}$  are asymptotic paths for the asymptotic values. 

 Pulling back, there is a preimage of $L_0$ containing  $\RR^-$; 
 label this preimage $L_{00}$. 
The boundary of $L_{00}$ is a bi-infinte curve both of whose ends at $-\infty$ are asymptotic to the horizontal lines $x \pm \pi i/4$. This implies the curve $C_0$  intersects $L_{00}$.   Set $L_{k0}=L_{00}+k\pi i$ 

A vertical line far to the left alternately intersects the preimages of $R_0$ and $L_0$.  The natural candidates for $R_{00}$ are the preimages of $R_0$ directly above and below $L_{00}$.    To be consistent with our convention for the addresses of the poles, if $\Im \la >0$, we label the preimage above $L_{00}$ by $R_{00}$ and if $\Im \la < 0$, we label the preimage below $L_{00}$ by $R_{00}$.    We will see that the case $\Im \la=0$ is special it won't matter which preimage we choose as $R_{00}$.  

The boundary of $R_{00}$ is  a bi-infinite curve in a left half plane symmetric about the horizontal line $x \pm i \pi /2$ where the sign depends on $\Im \la$.  It intersects $C_0$.  

\medskip

There are preimages of   $L_{k0}$ and $R_{k0}$ for each $k$ that intersect $L_{0}$; label them  $L_{0k0}$ and $R_{0k0}$. 
Inductively, going back and forth from $0$ to $-\infty$, using the inverse branches that take $L$ to $L_0$ and $L_0$ to $L_{00}$ and writing  $k_i$ for the non-zero indices, we obtain sets 
 $$ L_{0k_20\cdots k_n 0} \mbox{ and } R_{k_10k_20\cdots k_n 0}\subset L_{0\cdots k_{n-1}0} $$ at the origin intersecting $L_0$   and  assuming we are restricting the sets to their intersections with half-planes far enough to the left, we have sets 
$$ L_{k_10k_20\cdots k_n 0} \mbox{ and } R_{k_10k_20\cdots k_n 0}\subset L_{k_10\cdots k_{n-1}0} $$
intersecting the left half-plane and meeting at $-\infty$.  
Note that the relative ordering of the sets  at a given level is determined by the initial choice of $R_{00}$. 

Pulling back each of the sets $$ L_{0k_20\cdots k_n 0} \mbox{ and } R_{0k_20\cdots k_n 0}$$ to the pole $p_{k_1}$, we obtain sets $$ L_{k_1k_20\cdots k_n 0} \mbox{ and } R_{k_1k_20\cdots k_n 0}.$$  Then either 
$ R_{k_1k_20\cdots k_n 0}\subset L_{0\cdots k_{n-1}0} $ or $R_{k_1k_20\cdots k_n 0}$ and $R_{0\cdots k_{n-1}0} $   coincide. 
We will see later that the latter case occurs when there is an attracting periodic cycle.  
    
     These sets  in turn, can be pulled back to the prepole $p_{k_1k_2 \ldots k_l}$. 
 \begin{defn} We call the full collection of preasymptotic tracts
$$L_{k_1\cdots k_{l} k_{l+1} 0 k_{l+3}0\cdots k_{n-2} 0} \text{ and } R_{k_1\cdots k_{l} k_{l+1} 0 k_{l+3}0\cdots k_{n-2} 0} $$
the {\em L-R structure associated to $f$} and the subscript, $k_1\cdots k_{l} k_{l+1} 0 k_{l+3}0\cdots k_{n-2} 0 $ the {\em  index} of the set $L_{k_1\cdots k_{l} k_{l+1} 0 k_{l+3}0\cdots k_{n-2} 0} $.
\end{defn}

\begin{remark}\label{RLorder}  Since by our convention, if $\Im \la >0$,  $R_{00}$  is above $L_{00}$,  going counterclockwise, $L_{000}$ comes before $R_{000}$ and vice versa if $\Im t < 0$. Thus  the relative order of the $L_{0k}$ and $R_{0k}$ in $L$ determine the relative order of their preimages at $0$.   See figures~\ref{RLstruct} and~\ref{RLstruct1}.  Note that for this value of $\la$, the sets $R_{01k}$ coincide. 
\end{remark}

   \begin{figure}

\includegraphics[width=9in]{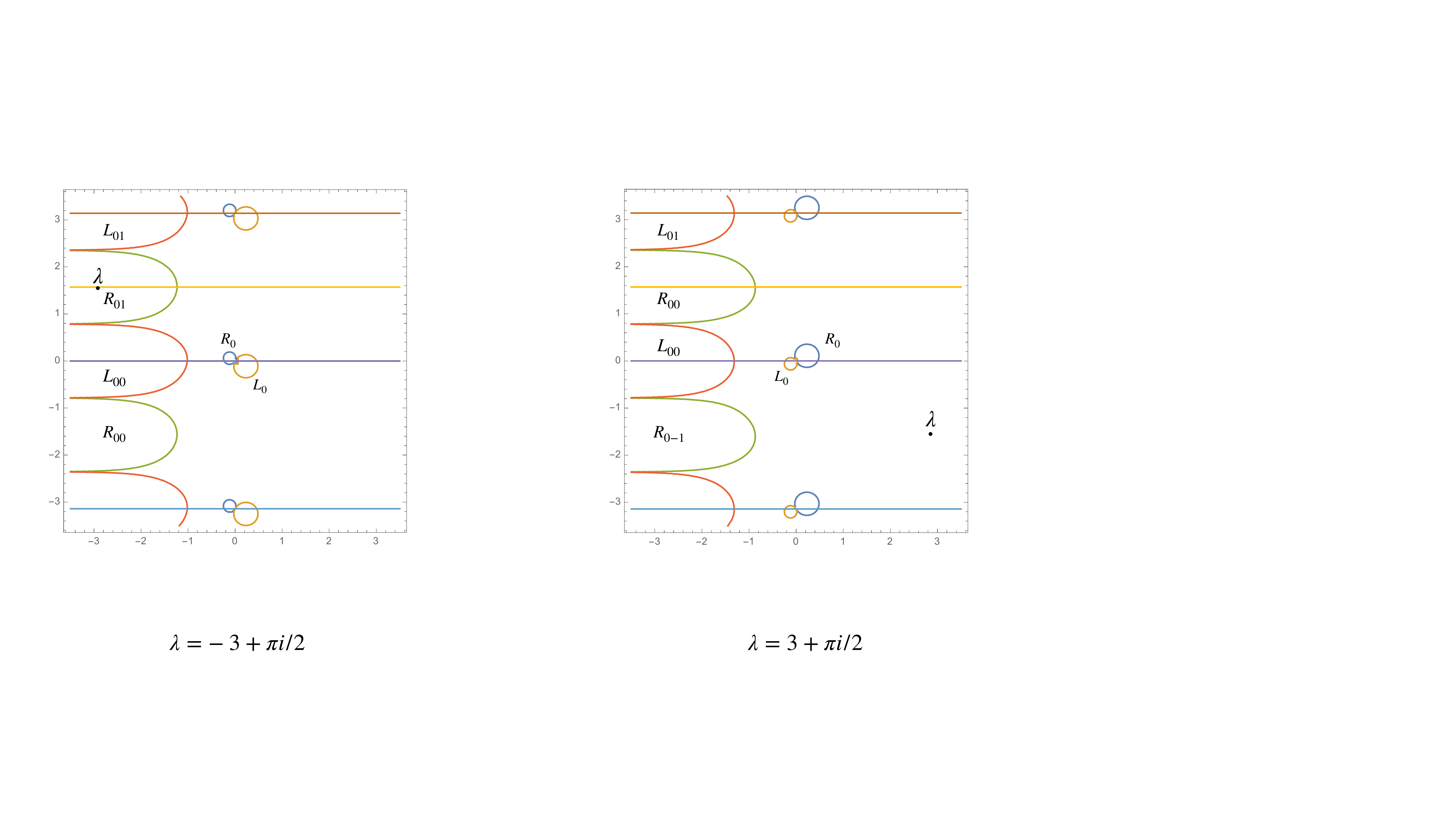}
\caption{ The first two pullbacks of the left and right halfplanes in the L-R structures of $f_{\la}$ for $\la=-3+i\pi/2$ on the left and $\la=-3+i\pi/2$ on the right. }

 \label{RLstruct}

\end{figure}

  \begin{figure}

\includegraphics[width=8in]{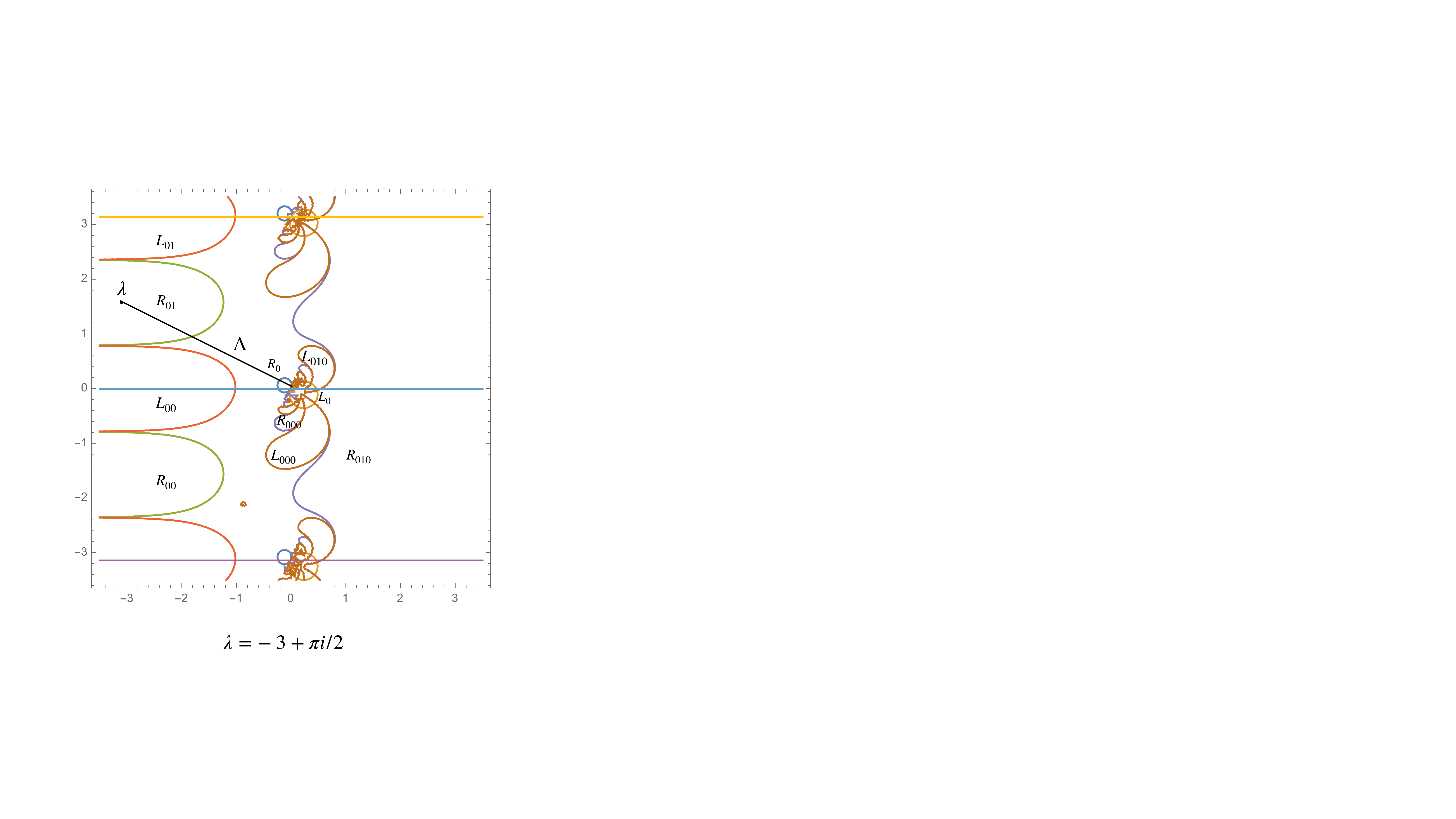}
\caption{ Three pullbacks of the left and right half plane in L-R structure of $f_{\la}$ for $\la=-3+i\pi/2$. }

 \label{RLstruct1}

\end{figure}

\subsection{Virtual cycles}\label{virtcyc}
If $v(\la)$ is a logarithmic asymptotic value of a holomorphic function $f$, it is locally omitted and some inverse branch  is logarithmic. Though it cannot belong to a  periodic  cycle,   $v(\la)$ may be part of one or more periodic cycles defined in a limiting sense; these limits are called a ``virtual cycles''.   This phenomenon has been studied in a number of previous works, in more or less generality (see \cite{ABF,FK, KK, CK1,CK2, CJK1,CJK2, CJK3}.

Here we give a definition in the context of the family $\calf_2$ of meromorphic functions with two asymptotic values and no critical values.  Note that both asymptotic values are omitted.

 \begin{defn}\label{stdvc} Suppose $f \in \calfp_2$ and  $\la$ is one of its asymptotic values.  If there exists an asymptotic path $\gamma_\lambda(t)$ of $\lambda$, such that $$\lim_{t\to\infty}f(\gamma_\lambda(t))=\lim_{t\to\infty}f^{n+2}(\gamma_\lambda(t))=\lambda,$$ then the set $\{ \la, f(\la), \ldots, f^{n-1}(\la), \infty \}$ is called a {\em virtual cycle for $f$ of period n+1}.  
\end{defn}

Suppose that $f^n(\lambda)=\infty$  and that $\delta_{\la}(t)$ is any asymptotic path  in the asymptotic tract of $\lambda$.   Let $\delta_1(t)$ be the component in the  set of preimages,  $f^{-n}(\delta_{\la}) $,  that has an endpoint at $\lambda$.    Any component of its preimage,  $\gamma_{\la}(t) \in f^{-1}(\delta_1(t))$,  lies in the asymptotic tract of $\la$  and satisfies $$\lim_{t\to \infty}f(\gamma_\lambda(t))=\lim_{t\to \infty}f^{n+2}(\gamma_\lambda(t))=\lambda.$$ That is, if $f^n(\lambda)=\infty$, then $\{ \la, f(\la), \ldots, f^{n-1}(\la), \infty \}$ is a virtual cycle. 

The condition in the definition above is that  the path $\delta(t)=f^{n+1}(\gamma_{\la}(t))$ is again an asymptotic path for $\la$.  It may happen however,   that $f^{n}(\la)=\infty$ but $\delta(t)=f^{n+1}(\gamma_{\la}(t))$ is an asymptotic path for the other asymptotic value  $\mu$ so that, in the limit, $f^{n+2}(\la)=\mu$.   If this does NOT happen, the cycle is called  {\em virtual cycle of minimal length} (see \cite{ABF}).   If it does, the virtual cycle may contain ``subvirtual cycles'', which makes the situation more complicated and more interesting as we will see below.    

In \cite{ABF,FK,KK}  virtual cycles occur for  functions in  families where periodic cycles have ``disappeared''; that is,  nearby functions in the family have  periodic cycles that become ``virtual''.   In the tangent family $\la \tan z$, for instance, suppose the parameter $\la_0=i \pi/2$.  Then if  $\la=\la_0+i\epsilon$, $\epsilon >0$, each asymptotic value $\pm i\la$ is attracted to a period two cycle of $f_{\la}$ whereas if   $\la=\la_0-i\epsilon$,  both asymptotic values are attracted to the same period four cycle of $f_{\la}$.  Thus, in the limit, as $\epsilon \to 0^+$, the period two cycles tend to a minimal virtual cycle but if $\epsilon \to 0^-$, the limit is a single (non-minimal) virtual cycle of period $4$. (See \cite{CJK1} where this phenomenon is called ``merging'').  We will make  this more precise for functions in  $\mathcal{FP}_2$ in section~\ref{paramsection}.    In the meantime, we continue our discussion of the dynamics of  functions in $\mathcal{FP}_2$ and leave the dependence on $\la$ for later.

\subsubsection{Virtual cycles in $\mathcal{FP}_2$}
\label{virtcycles}

Definition~\ref{stdvc} with $\la$ replaced by $0$ becomes 
\begin{defn}\label{defzeroinfcycle} Let $\gamma_0(t)$ be an asymptotic path for $0$ such that 
$$\lim_{t\to\infty} f(\gamma_0(t))=\lim_{t\to \infty}f^3(\gamma_0(t))=0.$$ Then $\{ -\infty, 0\}$ is a virtual cycle of period $2$ and is called the  {\em zero  virtual cycle}.   

\end{defn}
Every function in $\mathcal{FP}_2$ has a  zero virtual cycle.\footnote{Such cycles are called {\em persistent} in \cite{ABF}.}
\begin{prop}\label{0per2} Every $f \in \mathcal{FP}_2$ has a zero virtual cycle with minimal period $2$.  \end{prop}

\begin{proof}  Choose the asymptotic path $\gamma_0(t)$ inside the $L_{00}$ region of the L-R structure for $f$.  Then $f(\gamma_0(t)) \subset L_0$, $f^2 (\gamma_0(t)) \subset L$ and so is an asymptotic path for $0$. 
\end{proof}
 Note that if the asymptotic path $\gamma_0(t)$ is chosen inside the $R_{00}$ region of the L-R structure for $f$, $f^2 (\gamma_0(t)) \subset R$ and so is an asymptotic path for $\la$.

Although every $f \in \mathcal{FP}_2$ has a zero virtual cycle,  for most of these functions it is the only virtual cycle.  There are, however,  a set of values of $\la$ such that  $\{ \la, f(\la), \ldots, f^{n-1}(\la), \infty \}$  so that $f$ has another virtual cycle  whose  period is  minimal and is $n+1$.  
   As we will see in section~\ref{paramsection},  such an  $f$  also has virtual cycles of (non-minimal) periods $n+1+2m$ for every $m$.  Like the example $\pi/2 i \tan z$, these non-minimal cycles contain both asymptotic values  $\la$ and $0$.   The argument in the proof above implies that the difference between the minimal and non-minimal periods is always even because  the zero virtual cycle has period $2.$  This leads us to introduce the following definition.

 \begin{defn}\label{lavcs}  Suppose $\la$ is a prepole of $f=f_{\la}$ of order $n$ and  $\gamma_{\la}(t)$ is an asymptotic path for $\la$.
  \begin{enumerate}
  \item If $$\lim_{t \to \infty} f(\gamma_\lambda(t))=\lim_{t \to \infty}  f^{n+2}(\gamma_{\la}(t))=\lambda,$$  then the minimal cycle of period $n$,  $\{ \la, f(\la), \ldots , +\infty \}$  is called a {\em regular virtual cycle}.
    \item If there exists an $m$ such that $$\lim_{t \to \infty} f(\gamma_\lambda(t))=\lambda,$$ 
    $$\lim_{t \to\infty } f^{n+1+2j}(\gamma_{\la}(t))=-\infty,  \lim_{t \to \infty} f^{n+2+2j}(\gamma_{\la}(t))=0 $$ but $$ \lim_{t \to \infty} f^{n+2m}(\gamma_{\la}(t))=+\infty,$$
      then, as a limit, the cycle $$\{\la, f(\la) ,  \ldots , f^{n}(\la),   -\infty,0 , -\infty, \ldots,0, +\infty \}$$  is called an {\em $n+2m+1$ hybrid virtual cycle}.

    \end{enumerate}
    \end{defn} 
     Figures~\ref{per2vir} and~\ref{per4vir}  show schematics of virtual cycles of periods $2$ and $4$ at $\la = \pi i$.  
 \begin{figure}
    \centering\includegraphics[width=6in]{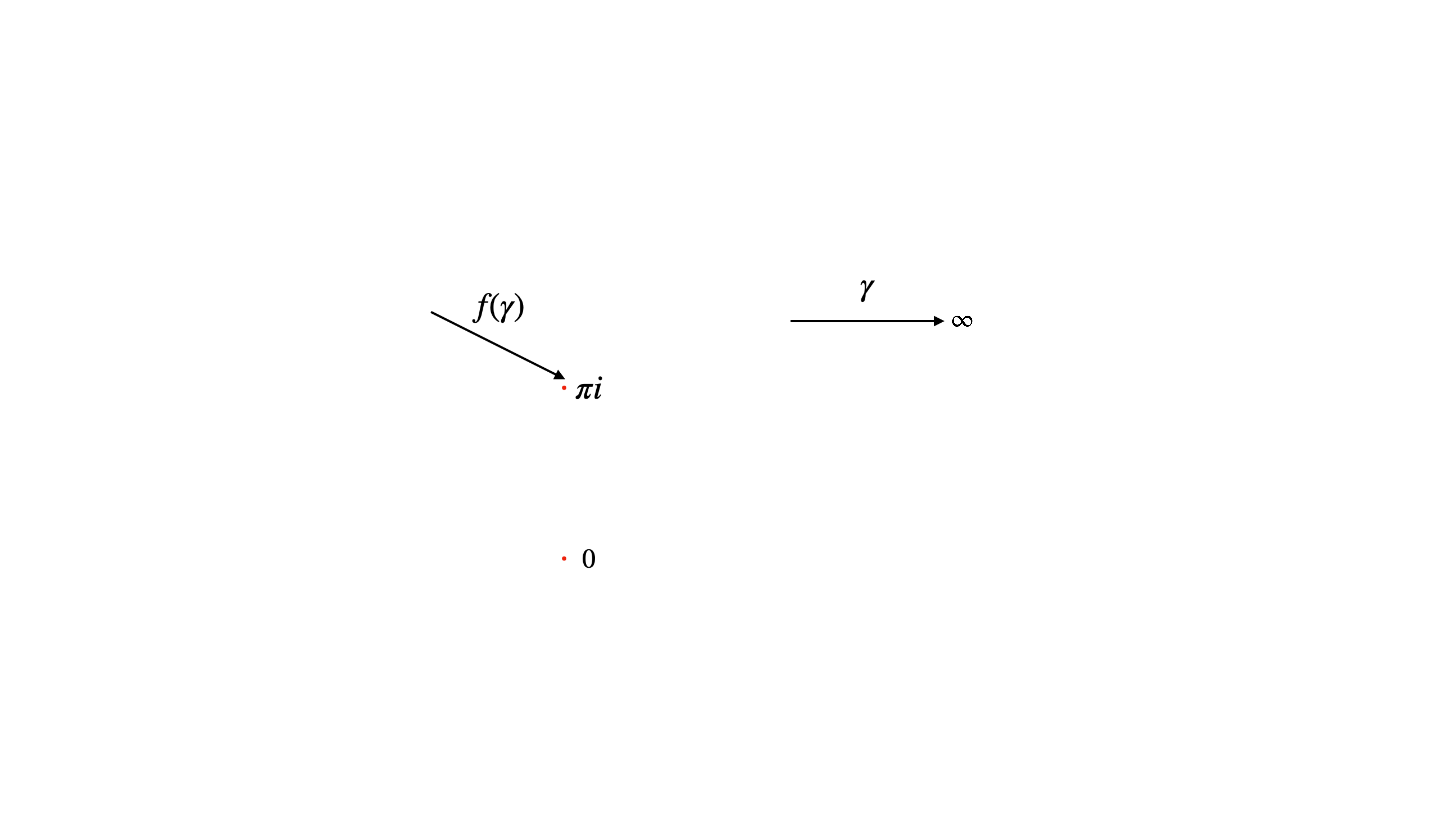}
     \caption{Virtual cycle period $2$ $\ \ \ \ \ \ \ \ \ \ \ \ \ \ \ $\label{per2vir}}

  \end{figure}
  \begin{figure}
    \includegraphics[width=6in]{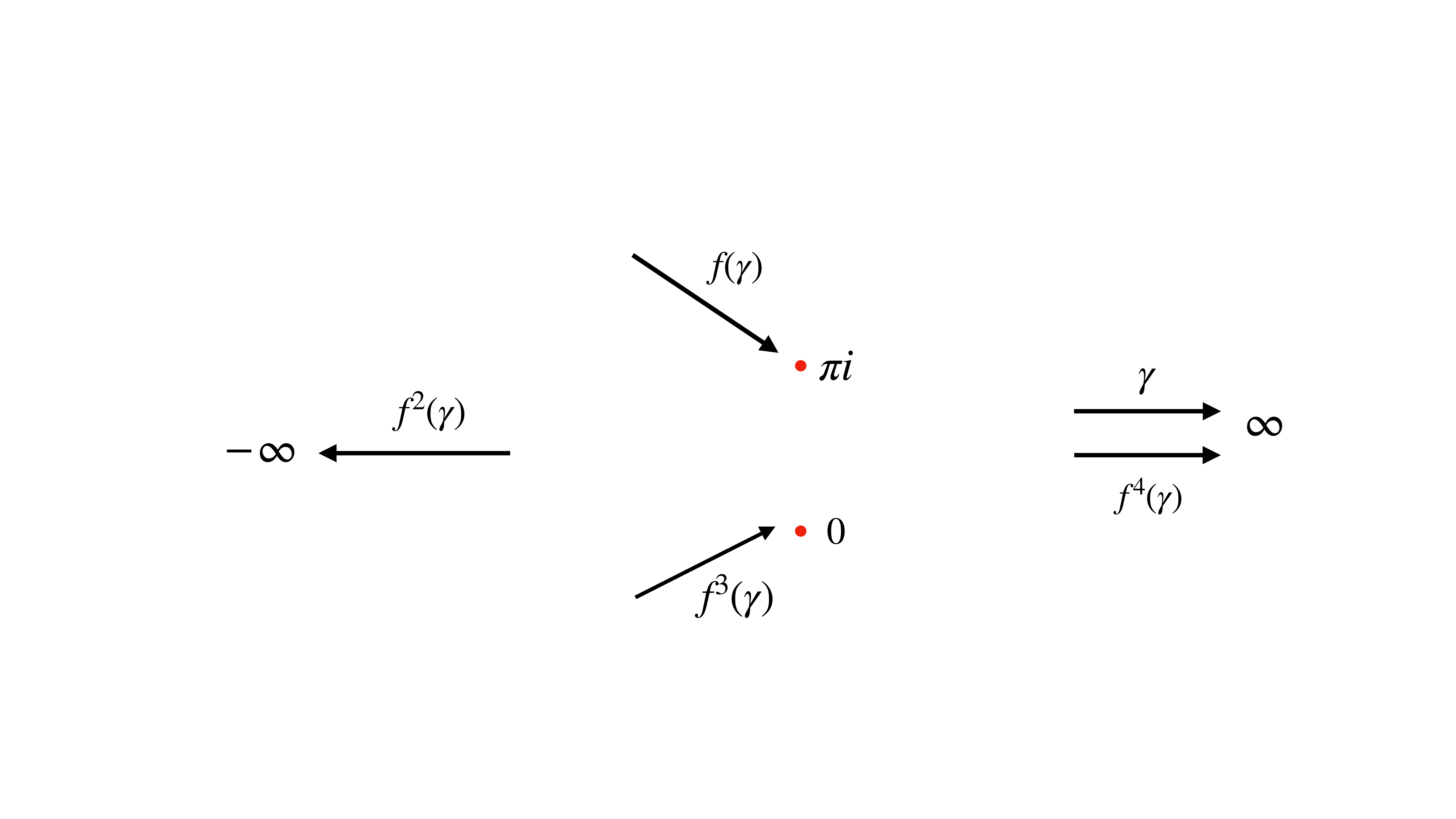}
     \caption{                                               Virtual cycle period $4$\label{per4vir}}

  \caption{}
  \label{virtualcycles1}
\end{figure}

Note that the points in the zero cycle are always in the Julia set and, if $f_{\la}$ has a virtual cycle containing $\la$, the points of its virtual cycle are also in the Julia set.

\subsubsection{Virtual multipliers}
\label{virtmult}
 The following estimates for $z$ near the zero cycle will used a number of times  throughout the paper:\\

 If $\Re z< M<<0, $
\begin{equation}\label{basic estimate}
f(z)= \lambda/(1-e^{-2z})\sim -\lambda/e^{-2z}\  {\mbox  and } \ f^2(z)\sim -e^{-2z}/2
\end{equation}
where the approximation depends on $M$ and is an equality in the limit as $M$ goes to $-\infty$.

If $\{z_0, z_1=f(z_0) \ldots,z_{n} =f^{n}(z_0)) \}$ is a periodic cycle of period $n+1$, its {\em multiplier} is $(f^{n})'(z_0)$.    Virtual cycles have ``virtual multipliers'' in the following sense:

\begin{defn}
If $\la$ is a prepole of order $n$ and $\gamma_\lambda(t)$ is an asymptotic path for $\la$, 
  set
$\calm(t)= (f^{n+1})'(f(\gamma_\la (t)))$.  Then if $\lim_{t \to \infty} \calm(t)$ exists, it is called the {\em virtual multiplier} of the cycle.  The multiplier of the zero virtual cycle is defined similarly.   
\end{defn}

\begin{prop}\label{zeroinfcycle}   $\lim_{t\to \infty} \calm(t)=0$.
\end{prop}
\begin{proof} For any $\lambda$,  if  $z$ is near $k\pi i$ for some $k\in \mathbb{Z}$, then
 $$f_\lambda(z)= \frac{\lambda}{1-e^{-2z}} \sim \frac{\lambda}{2(z-k\pi i)}, \, \, \, f'_\lambda(z)= \frac{-2\lambda e^{-2z}}{(1-e^{-2z})^2} \sim \frac{-\lambda}{2(z-k\pi i)^2},$$
 and $$f_\lambda'(f_\lambda(z))=\frac{-2\lambda e^{-2f_\lambda(z)}}{(1-e^{-2f_\lambda(z)})^2}.$$ Therefore,
 if $\Re (\lambda/(z-k\pi i))>>0$,   $$f_\lambda'(f_\lambda(z))\sim -2\lambda e^{-2f_\lambda(z)} \sim -2\lambda e^{\frac{-\lambda}{z-k\pi i}}$$ and  if $\Re (\lambda/(z-k\pi i))<<0$, $$f_\lambda'(f_\lambda(z))\sim -2\lambda e^{2f_\lambda(z)} \sim -2\lambda e^{\frac{\lambda}{z-k\pi i}}. $$

In either case \begin{equation}\label{supattracting} f'_\lambda(z)\cdot f'_\lambda(f_\lambda(z))\to 0, \ \text{as}\  \Re( \lambda/(z-k\pi i))\to \pm \infty. \end{equation}

Now suppose $\la \neq 0$ is a prepole of order $n-1$ with asymptotic path $\gamma_0(t)$.  For $j < n-2$,  there exist $k,K$ such that  $$0< k<|(f^j)'(\la)| < K< \infty. $$  Moreover, modifying $k,K$ if necessary, for $t$ sufficiently large,
$$0< k<|(f^j)'(\gamma_0(t))| < K< \infty. $$

If $\lambda$ is part of a regular virtual cycle, then $$\calm(t)=\prod_{i=1}^{n+1}f'(f^i(z)),$$ where the first $n-1$ terms of the product are bounded, and $f^{n+1}(\gamma_0(t))$ satisfies $\Re (\lambda/(f^{n+1}(\gamma_0(t))-k\pi i)>>0$. Therefore  $$\lim_{t\to \infty} \calm(t)=0.$$

If $\lambda$ is part of a hybrid virtual cycle, then $$\calm(t)=\prod_{i=1}^{n+2m+1}f'(f^i(z)),$$ where the first $n-1$ terms of the product are bounded, and the orbit of $\gamma_0(t)$ satisfies $\Re \lambda/(f^{n+1}(\gamma_0(t))-k\pi i)>>0$ and $\Re (\lambda/f^{n+2j+1}(\gamma_0(t))>>0$ for $j=1, \cdots m$. We have $$\lim_{t\to \infty} \calm(t)=0$$ as well. 
\end{proof}

\begin{remark} By virtue of this proposition, although $f_{\la}$ contains no critical values so it cannot have a super-attracting cycle, 
  these virtual cycles behave, in some sense, like {\em super-attracting} cycles. 
    \end{remark}
\begin{remark} In \cite{ABF} and \cite{FK} it is shown that virtual cycles often occur as limits of attracting cycles and that, in a limiting sense, the multipliers of these attracting cycles is $0$.\footnote{These are called {\em non-persistent} in \cite{ABF}.}  Although the virtual cycles do not have attracting basins, there are ``shadows'' of these basins in the following sense.   It follows from the results in \cite{Skor} that there are distinct completely invariant positive measure sets in the Julia set corresponding to each of the asymptotic values.
\end{remark}

 \section{Pseudo-hyperbolic maps}
 \label{p-h maps}
The post-singular set of a hyperbolic meromorphic map is disjoint from its Julia set.   Since one asymptotic value of the functions in $\calfp_2$ is always a pole and so belongs to the Julia set, none of    these functions is hyperbolic.  The following definition defines functions in the family that behave somewhat like hyperbolic maps.
\begin{defn} A map $f_\lambda \in \mathcal{FP}_2$ is called \emph{pseudo-hyperbolic} if it has an attracting periodic cycle $\mathbf z=\{z_0, \ldots z_{n-1} \}$ where $z_i=f_{\la}(z_{i-1})$,  $i=1,\cdots, n-1$, and $f_{\la}(z_{n-1})=z_0$.
\end{defn}
Recall that because $f_{\la}$ has no critical points, this cycle cannot be super-attracting.   Moreover, since $0$ is a pole, it cannot belong to an attracting cycle.

Assume $f_\la$ is  pseudo-hyperbolic and $\{ A_0, \ldots, A_{n-1}\}$ is the immediate basin of the attracting cycle $\mathbf z=\{z_0, \ldots z_{n-1} \}$ with $z_i \in A_i$.  As  the immediate basin must contain a singular value, and $0$ is a pole,   $\la \in A_i$ for some $i$.  Because it is open,  it also contains a neighborhood of $\lambda$.  Since $\la$ has no preimages,  the asymptotic tract of $\lambda$ is also in the immediate basin. By convention, we suppose $A_0$ contains the asymptotic tract of $\lambda$,  so that $\la \in A_1$.

\begin{prop}\label{access}  If $\la$ is pseudo-hyperbolic, the prepoles are all accessible boundary points of the full basin of the attracting cycle $\mathbf z=\{z_0, \ldots z_{n-1} \}$.  In particular,  every component of the basin whose boundary contains a given prepole contains a path that lands there.
 \end{prop}

 \begin{proof}  Let $\gamma(t)$ be an unbounded path  from   $z_0$ to $+\infty$  in $A_0$.  It lies in the asymptotic tract of $\la$ and lands at $+\infty$.  The successive  pull backs by $f$ are paths in the $A_j$ landing at preimages on their boundaries.  
 \end{proof}
 
 Recall the following terminology from \cite{DK,DT}.  
 \begin{defn}\label{CantorbouquetN} Let $\Sigma_N$ be the set of sequences  $\{k_0, k_1, \ldots, \}$ of integers  such that $|k_i|< N$. Let $U$ be open and connected and  suppose $f:U \rightarrow \CC$ is holomorphic.  An invariant subset $C_N$ of $J_f \subset U$ is called a {\em Cantor $N$-Bouquet} for $f$ if
 \begin{enumerate}
 \item  There is a homemorphism $\phi: \Sigma_N \times [0,\infty) \rightarrow C_N$.
 \item If $\pi: \Sigma_N \times [0,\infty) \rightarrow \Sigma_N$ is the projection map, set $\sigma(s)=\pi \circ \phi^{-1} \circ f \circ \phi(s,t)$.
 \item $\lim_{t \to \infty} \phi(s,t)=\infty.$
 \item $\lim+{n \to \infty} f^n \circ \phi(s,t)=\infty$ if  $t \neq 0$.
 \end{enumerate}
 \end{defn}
 There is a natural inclusion of an $N$ bouquet into an $N+1$ bouquet by considering only sequences whose entries have absolute value less than $N$.  The invariance of $C_N$ under $f$ requires that $f(\phi(s,0))=\phi(\sigma(s),0)$ so  $h(s,0)$ is an invariant set on which $f$ is topologically conjugate to the shift map.  
 \begin{defn} Let $C_N$ be a Cantor $N$-bouquet for $f$ and suppose $C_N \subset C_{N+1} \subset \ldots $ is an increasing sequence of bouquets with the natural inclusion maps.  The set 
 $$ C = \overline{ \cup_{N \geq 0} C_N }$$ is called a {\em Cantor Bouquet}.   
 \end{defn}
 
 The existence of the paths in the above proposition allow us to prove the first part of theorem A.

 \begin{thm}\label{CBouq}    If $\la$ is pseudo-hyperbolic, the Julia set always contains Cantor bouquets of curves that meet at the prepoles.
 \end{thm}
 \begin{proof} The arguments in \cite{DK} apply directly to show that these pullback paths can be used to  form Cantor $N$=bouquets for $f$.   Then, using the natural inclusion maps, we obtain Cantor bouquets. 
 \end{proof}

Before we talk  more about  pseudo-hyperbolic maps in general, we first consider the special case $n=1$.
\begin{prop}\label{compl invariance}  If $n=1$,  $A_0$ is  completely invariant. That is, $A_0$ is the only Fatou component.
\end{prop}
\begin{proof}
Because $n=1$ both the fixed point $z_0$ and $\lambda$ are in $A_0$. Let $U\subset A_0$ be a simply connected open set containing the fixed point $z_0$ and $\lambda$ (we can find this set by thickening a path from $\la$ to $z_0$). Then for every branch of the inverse,  $f^{-1}_\lambda(U)\subset A_0$ is a simply connected unbounded set containing all the preimages of $z_0$. Therefore $f^{-1}_\lambda(A_0)=A_0$;  that is,  $A_0$ is completely invariant.
\end{proof}

Thus,  applying theorem~\ref{CBouq}   we obtain the second half of theorem A.
\begin{cor}\label{Cantor Bouquet n=1} If $n=1$, the Fatou set is connected and the Julia set contains  Cantor bouquets at infinity and at each of the poles and prepoles.
\end{cor}
Note that because there is a right half-plane $R$ in $A_0$, the Cantor bouquet at infinity intersects a left half-plane.   This bouquet at infinity pulls back to Cantor bouquets at all the prepoles.  The preimages of $R$ are also in $A_0$ so the bouquets cannot intersect any of the R-sets in the L-R structure.
This is illustrated in Figure~\ref{dynpics1}.  
\medskip

Now suppose $n\geq 2$.  In the notation above,  $\lambda\in A_1$.  Then because our functions have only two asymptotic values, and one is never attracted to an attracting cycle, the proof, using  local linearization around the cycle, in   proposition 3.1 in \cite{FK} applies to show that 
 if $n>1$ all the components in the full basin of attraction are simply connected.

Moreover, the proof of proposition 3.2 in \cite{FK} also applies to show that $f: A_0 \rightarrow A_1 \setminus \{\la \}$ is a universal covering map, $A_0$ is unbounded and contains a right half-plane $R=\{z\ |\  \Re z>M\}$ for some $M>0$; all other maps $f: A_k \rightarrow A_{k+1}$ are one to one. Therefore each $A_k$, $k> 0$ is contained in a strip the  left half-plane with  height no more than $\pi$; that is, $A_k\subset  \{z\ |\ \Re z <0 \}$, and for  $z_1, z_2\in A_k$, $|\Im z_1-\Im z_2|<\pi$.  In particular, for those $j$ for which $A_j$ is  unbounded,   it is unbounded only to  the left; that is, there exists a sequence $z_m\in A_j$ such that $\Re z_m\to -\infty$ as $m\to \infty$.  Since such an $A_j$ is a pullback of $A_0$, going backwards around the cycle, it contains the pullback $R_{k_j \ldots k_{n-1}0}$ of $R$.

\subsection{Kneading sequences}\label{kneading}
In \cite{CJK2,FK,KK} it was shown that there is a fairly simple relationship between the inverse branches going around  the  cycles that attract the  ``free asymptotic value'' and the addresses of the  prepoles.
Because in $\calfp_2$, $0$ is not only a pole but also an asymptotic value of each function, this relationship is more complicated.

We  use the addresses of the prepoles and the  L-R structure for a pseudo-hyperbolic map to define a coding for the map that we call its kneading sequence.  Note that the definition of  kneading sequences here is not the same as the ones in \cite{D1,DFJ} although they are used  to prove similar kinds of results.

Assume $n>1$.  Suppose $f$ is pseudo-hyperbolic and let $A_n=A_0, \ldots, A_{n-1}$ be the components of the immediate basin of a periodic cycle $(z_0, \ldots, z_{n-1})$ labeled so that $\la \in A_1$.   Since $A_1$ is a component of $f^{-(n-1)}(A_0)$ and $R \subset A_0$ there is a unique set  in the L-R structure, $R_{k_1 \ldots k_n}$, whose index has length $n$  and is contained in $A_1$.

\begin{defn}\label{knseqnew} 
      We 
   define the {\em kneading sequence of $f_{\la}$} by the index of the set $R_{k_1 \ldots k_n}$:   $S=S(\la)=*k_1 \ldots k_{n-1}$. Note that if $n=1$, the sequence reduces to  $S=*$ since  $R$ has no index.
       \end{defn}
\begin{remark}\label{lainR}  We may assume without loss of generality that $z_0 \in R$ so that $z_1 \in R_{k_1 \ldots k_n}$ and that for all large $j$, $f^{jn}(\la) \in R_{k_1 \ldots k_n}$.  \end{remark}     

\begin{remark}  By the definition of the L-R structure, going forward, $f: R_{k_1 \ldots k_n} \rightarrow  R_{k_2 \ldots k_n}  $ and so on; that is $f$ acts as a shift map on the kneading sequence.    
\end{remark}

  We can relate specific prepoles on the boundaries of the $A_i$ to the kneading sequence.       
Since  $f$ is a homeomorphism, for each $A_i$, $i=2, \ldots n$,  there is a well defined inverse map  $g_i:A_i \rightarrow A_{i-1}$.

Because the prepoles in the boundaries of the Fatou components are accessible, $f$ extends continuously to them.    Therefore we can use the boundary point $+\infty$  of $A_0$ and an asymptotic path in $R$  to 
  define unique boundary points  $q_i$ on each $\partial A_i$, $i=0, n-1$ by going backwards around the cycle of domains as follows: set $q_0=q_n=+\infty$,  $q_{i-1}=g_i(q_i)$, $i=n-1, \ldots, 2$. Note that each $q_i$ is either a prepole, $0$ or $-\infty$.    Recall that each set at level $n$ of the L-R structure has a unique boundary point that is an $n^{th}$ preimage of $-\infty$ or $+\infty$.   Thus the points $q_i$ are boundary points of the sets $R_{k_i \ldots k_n}   \cap A_i$. 

  \begin{defn}\label{poleseq}
  The sequences $q_i$ fall into three categories.  Recall that $q_{n-1}$ is always a pole.
  \begin{itemize}
  \item {\rm case 1:}   $q_{n-1}=0$, and since $0$ is omitted,   $q_{n-2}=-\infty$.  Also,  for all $j=0, \ldots, (n-3)/2$,  $q_{n-{2j-1}}=0$ and $q_{n-2j}=-\infty$.  The sequence of $q_i$'s, the kneading sequence and the corresponding attractive cycle are all  called {\em unipolar}.  Note that unipolar sequences occur only for odd values of $n$.   
    \item {\rm case 2:}  $q_{n-1}$ is a non-zero pole. Its preimages  $q_i$ are all prepoles whose order increases as $i$ decreases.   This sequence of $q_i$'s, and the corresponding  kneading sequence and  attractive cycle are called {\em regular}. 
    \item {\rm case 3:}  $q_{n-1}=0$, and since $0$ is omitted,   $q_{n-2}=-\infty$.   In this case,  for $j=0, \ldots, K$, $K < (n-1)/2$,   $q_{n-{2j+1}}=0$ and $q_{n-2j}=-\infty$,  but $q_{n-{2K-3}}$ is a non-zero pole.  The remaining $q_i$ are prepoles whose order increases as $i$ decreases.    This sequence of $q_i$, and the corresponding  kneading sequence and  attractive cycle are called  {\em hybrid}.    
 \end{itemize}
 \end{defn}

   The next several lemmas contain the technical results we need to prove theorem B. In all of them $f$ is assumed pseudo-hyperbolic and $n>1$.

Figures~\ref{dynpics1} and ~\ref{dynpics2} illustrate the different cases in the lemmas.

In the next lemma we assume $\partial A_0$  contains a pole.  Note that if it contains one pole, by periodicity it contains them all.

\begin{lemma}\label{separating}
If $0\in \partial A_0$ and the period $n$ of the attracting cycle is  greater than $1$,  then the component $A_1$ is unbounded,  $n$ is odd and the kneading sequence is $*k_10\cdots 0k_{n-2 }0$, where for some odd $j$,  $k_{n_j} \neq 0$.

\end{lemma}

\begin{proof}
Let $V$ be a small neighborhood of $0$ and let $U_0=V \cap  A_0$.  Then $0$ is a boundary point of $U_0$ and since $f(U_0) \cap A_0 =\emptyset$, $U_0 \cap R_0 = \emptyset$ and $U_0 \cap L_0 \neq \emptyset$.

 By  proposition~\ref{access}, there are paths  in $A_0$ that land on $0$ and $+\infty$ respectively.  Since $A_0$ is connected they can be joined to form a path $\gamma_0 \subset A_0$ that goes from  $+\infty $ through R  and  $L_0 \cap A_0$ to $0$.  
  Because  the pole $0\in \partial A_0$,  going forward around the cycle, and setting $\gamma_i=f(\gamma_{i-1})$, $i=1, \ldots, n$,  $$\lim_{t\to \infty} \gamma_1(t)=\lim_{t\to \infty} f(\gamma_0(t))\to \infty.$$   If $\lim_{t\to \infty}\Re \gamma_1(t)=+\infty$,  for $t$ close to $\infty$, $\gamma_1 \cap A_0 \neq \emptyset$ and $A_1=A_0$, which can't happen since $n>1$.  Therefore  $\lim_{t\to 0} \Re \gamma_1(t)= \la,$ and $\lim_{t\to \infty}\Re \gamma_1(t)=-\infty$ so that $A_1$ is unbounded to the left and intersects  an asymptotic tract for $0$;  that is,  $\gamma_1(t)$ goes from $\la$  to $-\infty$.

  Applying $f$ to $A_1$,  $-\infty \in \partial A_1$ implies $0 \in \partial A_2$.    Continuing forward  around the cycle we see that for $j$ odd, $A_j$ is unbounded and for $j$ even, $0 \in \partial A_j$.  The endpoints of the paths $\gamma_j(t)=f(\gamma_{j-1}(t))$ thus alternate between $0$ and $-\infty$ until   $j=n$ and $\lim_{t \to \infty}\gamma_{n}(t)=+\infty$;  thus $n$ must be odd.

  Now go backward around the cycle and set $\beta_{n-1}=g_{n-1}(\gamma_0)$ and $\beta_{n-j-1}=g_{n-j}(\beta_{n-j})$, $j=2, \ldots, n-1$.  
   Denote the successive preimages of $\gamma_0$ by $\beta_{n-j}$ so that the domain $A_{n-j}$ contains the curve $\beta_{n-j}$.  Note that each of these curves passes through a periodic point and joins a prepole to either $0$ or $\infty$ depending on the parity of $j$.  The domains $A_{n-j}$ are unbounded on the left and are disjoint so their intersections with a vertical line determines an ordering of  the imaginary parts of $\beta_{n-j}$.

\begin{figure}
\begin{center}

\includegraphics[width=4.5in]{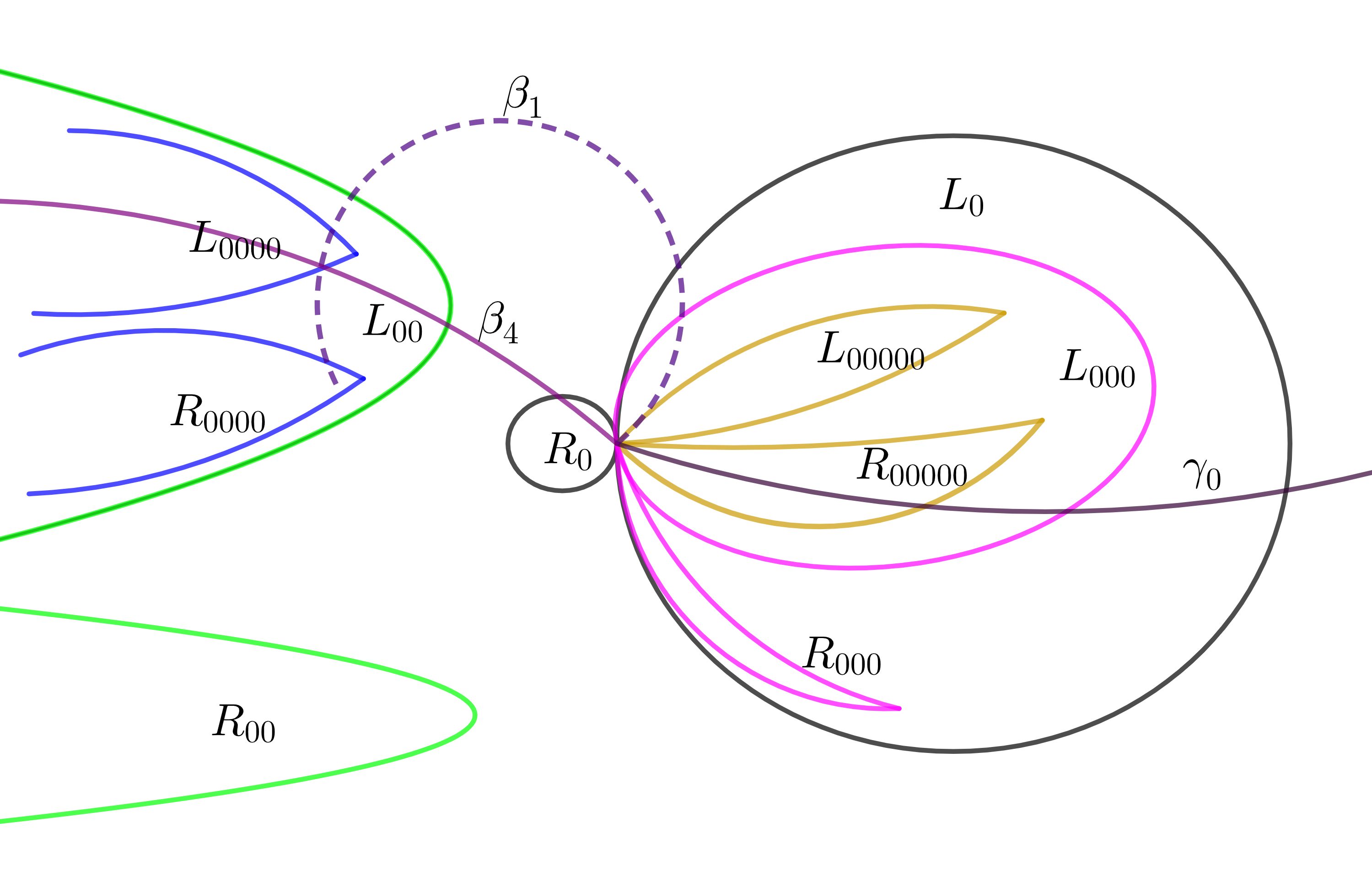}
\caption{ A period $5$ example that shows the kneading sequence cannot be $*0000$.   }

\label{NoZeros}
\end{center}   
\end{figure}

  We next claim $k_{n-2j} \neq 0$ for some $j=1, \ldots (n-1)/2$.   By the definition of the kneading sequence, the part of $\gamma_0$ in $L_0$ approaches $0$ through $R_{k_10\cdots 0k_{n-2 }0}$. Pulling back,   $\beta_{n-1}$ in $A_{n-1}$   goes from $L_{k_20 \ldots k_{n-2}0} \subset L_{00}$ to $R_0$.  By remark~\ref{RLorder}, the regions  $L_{00}$ and  $R_{00}$ intersect a vertical line in order depending on the argument of $\la$.  Pulling back the vertical direction to $0$, the relative order of the regions $R_{00 \ldots 0}$ and $L_{00 \ldots 0}$,  where the number of zeros is between $1$ and $n-1$, is unchanged.  Suppose for argument's sake that $L_{00}$ lies above $R_{00}$ as in figure~\ref{NoZeros}.   Otherwise the figure is reflected in the real axis.
  On the left, the $L_{0 \ldots 0}$'s have an even number of indices and are nested.  The region $R_{0 \ldots 0}$ with $2j$ zeros lies below the region $L_{0 \ldots 0}$ with $2j$ zeros and, far enough to the left, is inside the region $L_{0 \ldots 0}$ with $2j-2$ zeros.

  If all the $k_{n-2j}$ were $0$, $\gamma_{0}$ would go from $R_{00 \ldots 00}$ to $R$, where the number of $0$'s is $n-1$.   Similarly
   $\beta_{n-1}$  would start in $L_{00 \ldots 0}$, where the number of $0$'s is $n-2$ and end in $R_0$.  Thus the union of the curves $\gamma_0 \cup \beta_{n-1}$ would divide the plane.    In particular, on the left it would separate  the regions $R_{0\cdots 0}$, with $n-1$  $0$'s and $L_{0\cdots 0}$, with $n$  $0$'s.   Both these regions must be in $A_1$, however, so the assumption that all $k_{n-2j}=0$ implies $A_1=A_{n-1}$ and $n=2$, a contradiction.  Therefore     one of the $\beta_{n-2j}$   must join $-\infty$ to the non-zero pole $k_{n-2j} \pi i$.
\end{proof}

The following lemma will be proved in section~\ref{zeronotvc} but we state it here to complete the ideas of this section.
 \begin{lemma}\label{evenper}  If $n=2m$, $m>0$, the kneading sequence cannot have the form $S=*0k_20 \ldots 0 k_{2m-2}0$; in particular, if $n=2$, $S=*k_1$, $k_1 \neq 0$.
 \end{lemma}

Similar reasoning gives us
\begin{lemma}\label{regseq}
If the sequence of $q_i$'s for the attractive cycle is regular, all of the $A_i$, $i=1, \ldots n-1$ are bounded.   Moreover, there are no prepoles on $\partial A_i$ of order less than $n-i$. The kneading sequence has the form $k_1 \cdots k_n$ where for all $i$, $k_i \neq 0$.
\end{lemma}

\begin{proof}  If $A_{n-1}$ is unbounded, then because $n>1$, it intersects $L$, and as above, $0 \in \partial A_0$ and by periodicity, all the poles are in $\partial A_0$.  In particular, since $q_{n-1}$ is a non-zero pole, $k_{n-1}\pi i$,  $\partial A_{n-1}$ contains two poles.  We saw above that $A_0$ cannot intersect any of the sets $R_k$ so that $A_{n-1}$ must intersect both the sets $R_0$ and $R_{k_{n-1}}$.  Since $f$ is periodic,   this implies $f:A_{n-1} \rightarrow A_0$ is not a homeomorphism.  Going backwards around the cycle, the rest of the $A_i$, $i>0$ are bounded as well.   

Now $A_{n-1}$ is bounded and has only one pole, $k_{n-1}\pi i$, on its boundary and it intersects $R_{k_{n-1}}$.   By hypothesis, $q_{n-2}$ is a  prepole of order $2$ so it is finite.  Since that is the only pole on $\partial A_{n-1}$ and $A_{n-1}$ is bounded, all the prepoles on  $\partial A_{n-2}$ are of order at least $2$ and thus are finite which implies $A_{n-2}$ is bounded.   Continuing backwards around the cycle, all the $A_i$, $i>0$ are bounded.   The sets  $R_{k_i \cdots k_{n-1}}$ are all bounded sets as well since if any of the $k_i = 0$, then $q_i =0$ and $A_{i-1}$ would not be bounded. 
 \end{proof}
 
 The ideas in the proofs above combine to prove:
 
 \begin{lemma}\label{hybseq} If the sequence of $q_i$'s for the attractive cycle is hybrid, all of the $A_i$, $i=1, \ldots, K$ are bounded and the rest of the components are alternately unbounded or have $0$ on their boundary.
 \end{lemma}

   The lemmas above show that if $\la$ is pseudo-hyperbolic and $n>1$, not all entries in the kneading sequence can be zero, and that if $n$ is even, the sequence cannot have the form $*0k_20 \ldots 0k_{n-1}0$.    Thus we have
\begin{defn}\label{defn9} A kneading sequence $*k_{1}k_{2} \ldots k_n$ is {\em allowable} if it corresponds to a unipolar, regular or hybrid sequence of $q_i'$.  In particular, if   $n>1$,  all sequences are allowable as kneading sequences EXCEPT those such that all entries are  zero and such that   if $n$ is even, the sequence cannot have the form $*0k_20 \ldots 0k_{n-1}0$.
\end{defn}

The lemmas immediately imply

\begin{thm}[Theorem B]\label{thmB}
Let $f_\lambda $ be a pseudo-hyperbolic map whose attracting cycle has period $n$. Then it has a unique kneading sequence $S=*k_1\cdots k_{n-1}$ that  is either regular,  unipolar or hybrid.
\end{thm}

\begin{remark}
In the dynamics of regular pseudo-hyperbolic maps  the polar asymptotic value $0$ is not on the boundary of any of the components of the immediate basin.  It is therefore  similar to the dynamics of hyperbolic maps investigated in earlier work on purely meromorphic functions, all of whose asymptotic values are finite, see \cite{KK, DK, FK, CJK1, CJK2, CJK2, CK1, CK2}.     This  is not true for unipolar or hybrid pseudo-hyperbolic maps since $0$ is always on the boundary of $A_{n-1}$.  For these maps, the dynamics  also exhibit  properties similar to those of exponential maps (see e.g. \cite{DFJ,S}); note that for these exponential  maps, one asymptotic value is infinity --- and thus a pole of order $0$.    This similarity
  is particularly evident when
  the sequence is unipolar.    In our discussion of parameter space below, it will be useful to separate the unipolar  from the regular and hybrid  cases.
  \end{remark}

\subsection{Examples}
In this section we show examples of the dynamical plane for functions with each of the types of kneading sequences.   In  Figure~\ref{dynpics1},  the figure on the left is for a period 1 function  with $S=*$; in the color version, the yellow area is the Fatou set.  The figure on the right is in a regular period 3 component with kneading sequence $*-2-1$ and the Fatou set is shown in red in the color version.   Figure~\ref{dynpics2} shows  two period 5 examples; on the left is a unipolar example with kneading sequence $*0020$  and on the right is a hybrid example with sequence $*1110$.  In both, the Fatou set is shown in green.  In each figure the poles are visible as small disks.  This is due to roundoff.

\begin{figure}[ht]
  \centering
  \begin{subfigure}[b]{0.5\linewidth}
    \centering\includegraphics[width=5.8cm]{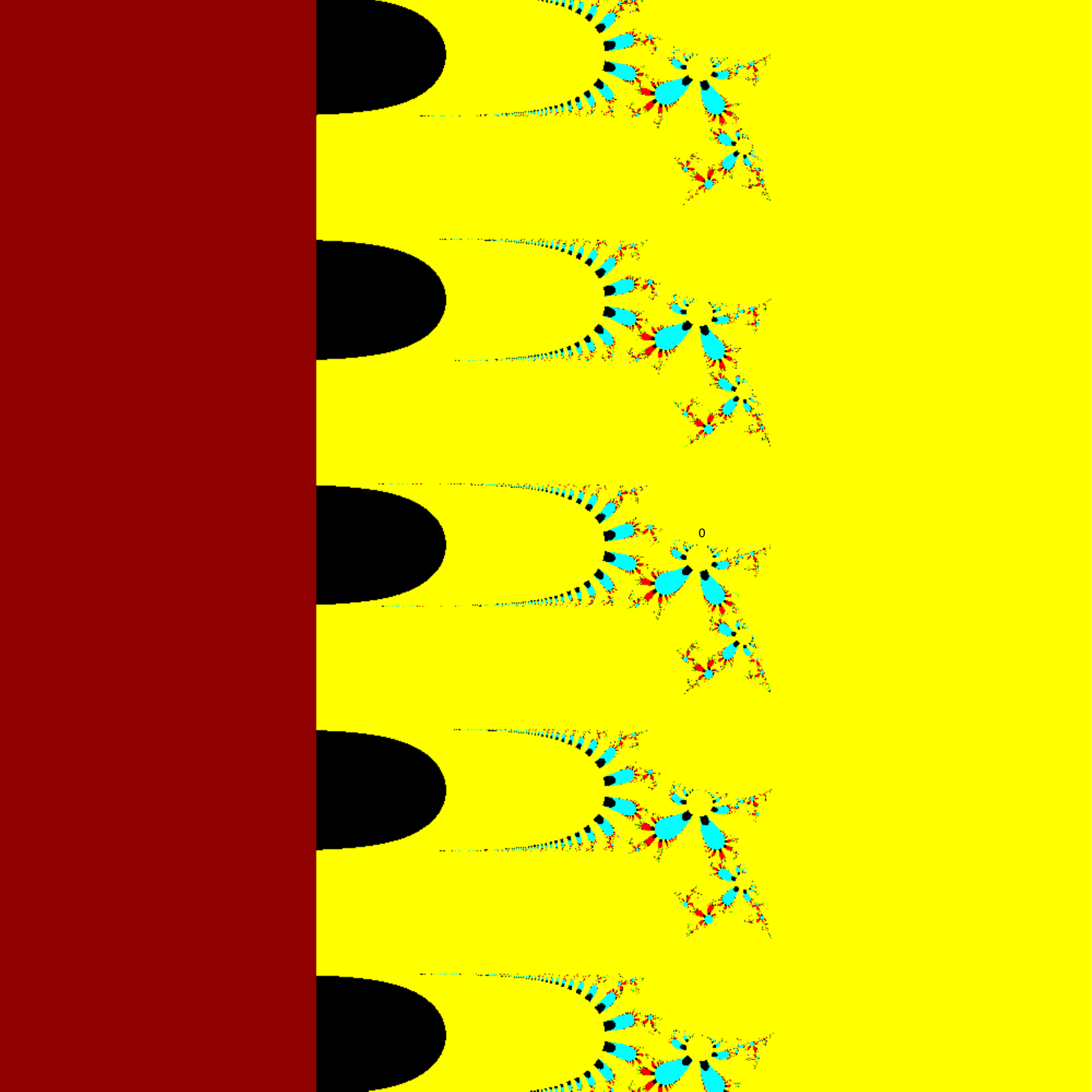}
     \caption{Regular Period $1$ $\ \ \ \ \ \ \ \ \ \ \ \ \ \ \ $\label{fig:fig1}}

  \end{subfigure}%
  \begin{subfigure}[b]{0.5\linewidth}
    \centering\includegraphics[width=7.6cm]{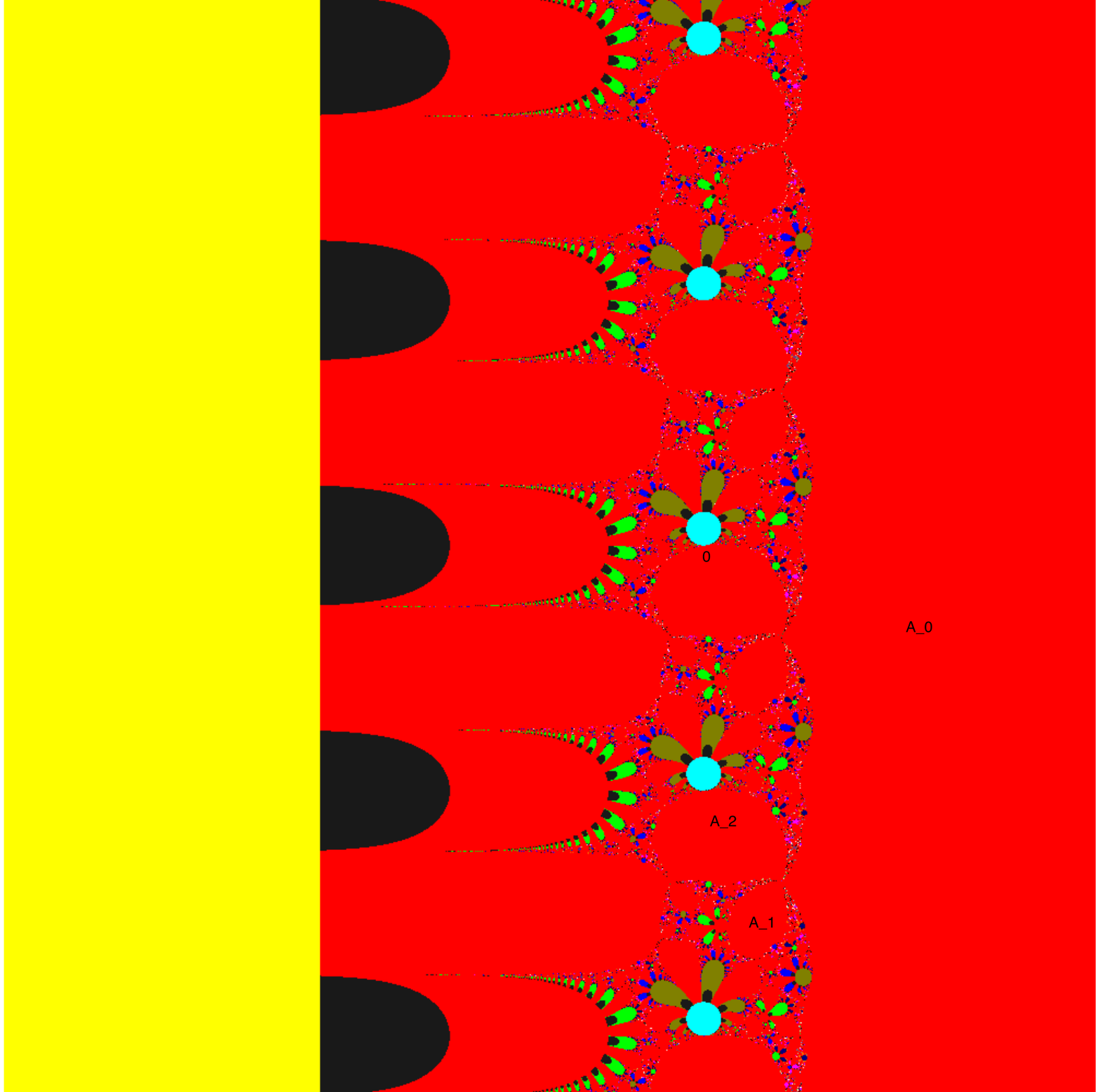}
     \caption{                                                 Regular Period $3$\label{fig:fig2}}
    \end{subfigure}

  \caption{}
  \label{dynpics1}
\end{figure}

\begin{figure}[ht]
  \centering
  \begin{subfigure}[b]{0.5\linewidth}
    \centering\includegraphics[width=5.7cm]{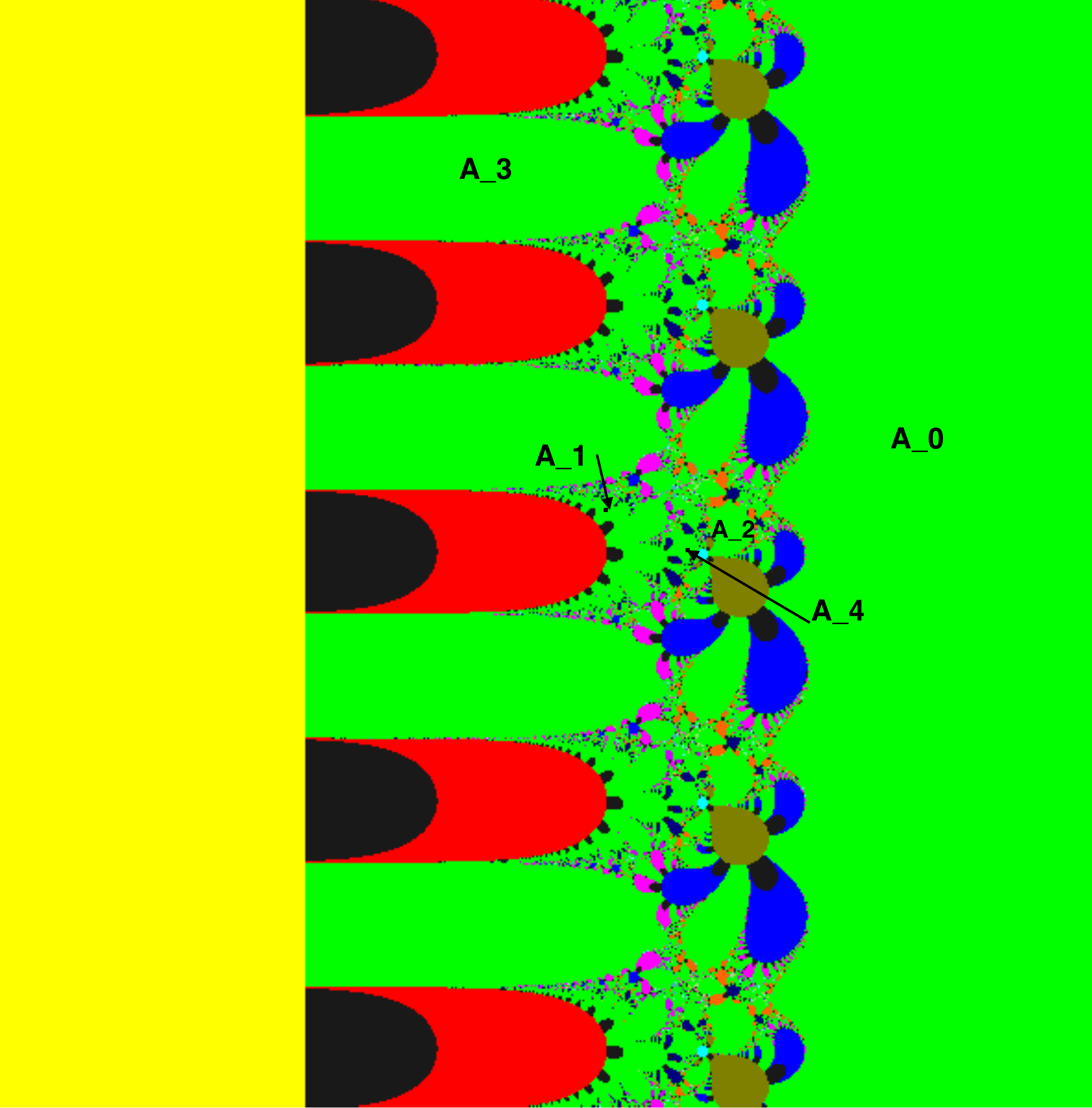}
    \caption{Unipolar period $5$\label{fig:fig1}}
  \end{subfigure}%
  \begin{subfigure}[b]{0.5\linewidth}
    \centering\includegraphics[width=7.4cm]{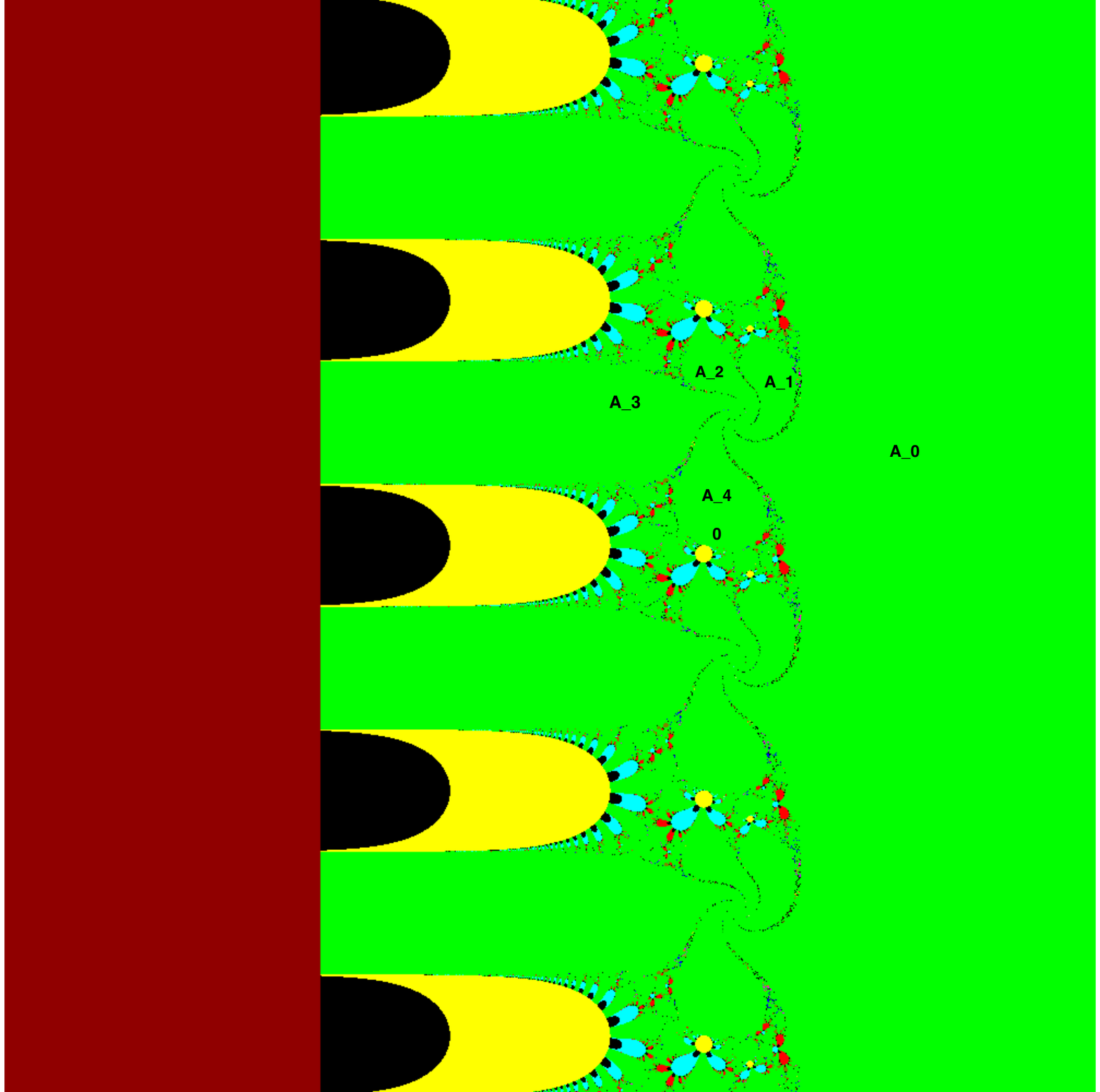}
    \caption{Hybrid period $5$\label{fig:fig2}}
  \end{subfigure}
  \caption{}
   \label{dynpics2}
\end{figure}

\section{The Parameter plane of $\mathcal{FP}_2$}
\label{paramsection}
We now turn our attention to how the dynamics change as we vary the parameter
  $\lambda\in \mathbb{C}^*$.   Note that $\la=0$ is a singularity since $f_{\la}$ is not defined.

\begin{lemma}
The parameter space is symmetric with respect to the real line.
\end{lemma}

\begin{proof}  It is easy to check that $$\overline{f_\lambda(z)}=f_{\overline{\lambda}}(\overline{z}),$$ that is $f_\lambda(z)$ and $f_{\overline{\lambda}}(z)$ are conjugate.  \end{proof}

Before we study the properties of the parameter plane we prove two propositions about real values of $\la$.
We prove that if $\la \in \RR^+$, then $\la$ is attracted to a real fixed point  and if 
 $\la \in \RR^-$,  the orbit of $\la$ accumulates on the zero virtual cycle and  thus is not pseudo-hyperbolic.

\begin{prop}\label{posla} 
The positive real line $\lambda =x, x>0 $ is contained in  a component of   $\mathcal H_1$ of the parameter plane where $f_{\la}$ has an attracting fixed point. 

\end{prop}
\begin{proof}
For any real $\lambda>0$,  $$\lim_{x\to 0^+}f_\lambda(x)=\infty \mbox{  and } \lim_{x\to \infty} f_\lambda(x)=\lambda.$$ Therefore the equation $f_\lambda(x)=x$, written out as
 $$\frac{\lambda}{1-e^{-2x}}=x,$$ always has a positive solution.  Rewriting as  $\lambda=x(1-e^{-2x})$, we see that the fixed point $x$ is increasing with respect to $\lambda$. In particular, $x\to 0$ as $\lambda\to 0^+$ and $x\to \infty $ as $\lambda\to \infty$.

As above, the multiplier of the fixed point is $$\frac{-2\lambda e^{-2x}}{(1-e^{-2x})^2}=\frac{-2xe^{-2x}}{1-e^{-2x}}.$$ It tends to $-1$ as $x\to 0^+$ and to $0$ as $x\to \infty$.
\end{proof}

\begin{prop}\label{negescape}
If $\lambda<0$, the asymptotic value $\lambda$ is attracted by the  zero virtual cycle;  that is, $\RR^-$ does not intersect any pseudo-hyperbolic component.
\end{prop}

\begin{proof}
For any $\lambda<0$ and $x$ real, the function $f_\lambda(x)$ is increasing, $\lim_{t \to 0^+}f_\lambda(t)=-\infty$, $\lim_{t \to 0^-}f_\lambda(t)=+\infty$, $\lim_{t \to -\infty} f_\lambda( t)=0$ and $\lim_{t \to +\infty} f_\lambda( t)=\la$.  To show that $\lambda$ is attracted by the zero cycle, we only need to show $f^2_\lambda(x)<x$ for any $x>0$.

It is easy to check that  for any $t\neq 0$,
 \begin{equation}\label{ref}1-e^{-t}<t . \end{equation}
 This inequality   and $\la <0$,
imply that for any $x>0$,  $$f_\lambda(x)=\frac{\lambda}{1-e^{-2x}}<\frac{\lambda}{2x}<0.$$
 Since $f_\lambda(x)$ is increasing, it follows that
  \begin{equation}\label{mid} f_\lambda(f_\lambda(x))<f_\lambda(\frac{\lambda}{2x})=\frac{\lambda}{1-e^{-\frac{\lambda}{x}}} \end{equation}

Let $t=\lambda/x$, and apply inequality (\ref{ref}) again, to get $$1-e^{-\frac{\lambda}{x}}<\frac{\lambda}{x}.$$  Since $$1-e^{-\frac{\lambda}{x}}<0,$$  it follows that $$\frac{\lambda}{1-e^{-\frac{\lambda}{x}}}<x$$

Together with inequality (\ref{mid}), this implies that $f^2_\lambda(x)<x$ for all $x>0$.
\end{proof}

 In section~\ref{virtcyc} we saw that there are functions in $\mathcal{FP}_2$ for which the asymptotic value $\la$ is a prepole and is thus part of a regular or hybrid virtual cycle.  They are of particular interest in our study of the parameter plane so we give them a name.
\begin{defn} If the asymptotic value $\la$ of  $f_{\la}$ is a prepole, and hence part of a regular or hybrid virtual cycle, we call $\la$ a {\em virtual cycle parameter}.
\end{defn}

\subsection{Pseudo-hyperbolic components}
In the tangent family studied in \cite{KK} or the families studied in \cite{CK1,FK}, the parameter space contains open sets of topologically conjugate hyperbolic maps separated by a bifurcation locus.  In the earlier work,  \cite{FK,KK}, the  hyperbolic components  that are universal covering maps of the punctured disk under the multiplier map are called {\em shell components} because of their shape; and their  properties were studied  in detail.

Since functions in $\calfp_2$ do not have critical points, the multiplier map which maps the multiplier of the attracting cycle of a point  in a pseudo-hyperbolic component into the unit disk cannot take the value $0$.
\begin{defn}
Let $\mathcal H_n$ be the set of all $\lambda$ such that $f_\lambda$ has an attracting cycle of period $n$ and denote any connected subset of $\mathcal H_n$ by $\Omega_n$.   We use the terms \emph{ pseudo-hyperbolic component} or {\em shell component} interchangeably to describe   $\Omega_n$.
\end{defn}

\subsection{The set $\mathcal H_1$}
\label{H1}

Proposition~\ref{posla} implies 
the positive real line $\lambda =x, x>0 $ is contained in   $\mathcal H_1$.

\begin{remark} On the positive real line, the multiplier of the fixed point is between $0$ and $-1$. In particular, the multiplier approaches $-1$ as $\lambda$ approaches $0$.  There is, however,  no period doubling at $0$ at this point because it is a singularity of the parameter space.
\end{remark}

\begin{thm}\label{Omega1}
The set  $\mathcal H_1$ consists of a single connected component $\Omega_1$ which  is an unbounded simply connected open set intersecting  the right half-plane. 
The point $+\infty$ acts as its  virtual center.
\end{thm}
See the component on the right in figure~\ref{param}.
\begin{proof}
 For any $\lambda$, any fixed point $z=z(\la)$ of the map $f_\lambda$ satisfies the equation $$\frac{\lambda}{1-e^{-2z}}=z$$ and its multiplier is $$\calm(\la)=\calm(z(\la))=\frac{-2\lambda e^{-2z}}{(1-e^{-2z})^2}=\frac{-2z e^{-2z}}{1-e^{-2z}}.$$

It is easy to check that $$\lim_{\Re z\to +\infty} \frac{-2ze^{-2z}}{1-e^{-2z}}=0$$ and $$\lim_{\Re z\to -\infty} \frac{-2ze^{-2z}}{1-e^{-2z}}=\infty.$$

Therefore $\calm(z)$ has two asymptotic values, $0$ and $\infty$,   with  asymptotic tracts in the right and left half-planes respectively.  Moreover, $0$ is an omitted value. Therefore, for a small punctured neighborhood $U$ of $0$, $\mathcal{M}^{-1}(U)$ is an asymptotic tract for $0$ and is thus an unbounded connected set in the right half-plane; extending by analytic continuation shows  that $\mathcal{M}^{-1}(\mathbb{D}^*)$ is connected as well.

Since $\calh_1=\la(\calm^{-1}(\DD^*))$ is the holomorphic image of a connected set,  it consists of a single component $\Omega_1$.

Finally, since $\lambda=z(1-e^{-2z})$ and   $$\frac{\lambda}{z}=1-e^{-2z}\to 1,$$  as $\Re z\to \infty$, the real parts of  both the fixed point and  $\la$   tend to $+\infty$.   In particular,
$\Omega_1$ has a boundary point at $+\infty$ where the limit of $\calm$ is zero, so this acts like a virtual center.
\end{proof}

\subsection{Properties of $\Omega_n$}
The following proposition combines several results proved in  \cite{ABF,FK} adapted to our situation.

\begin{thm}\label{omeganprops}  Let $\Omega_n$ be a component of $\calh_n$. Then
\begin{enumerate}
\item
 The multiplier map $\calm$ that sends a point $\la \in \Omega_n$ to the multiplier of the attracting cycle of $f_{\la}$ is a universal cover of the  punctured disk.  
 \item  The boundary of $\Omega_n$ is analytic along curves where $\calm(\la)=e^{i\theta}$, $\theta \neq 2\pi k$ for any integer $k$ and it has a cusp if  $\theta = 2\pi k$.  The component thus has the appearance of a shell, hence the name. 
 \item
 All the maps in a $\Omega_n$ are quasi-conformally and hence topologically conjugate.
\item\label{dd}
 There is a unique boundary point $\la^*$ of $\Omega_n$ and a path $\la(t) \in \Omega_n$ ending at $\la^*$ such that $\lim_{t \to 1} \calm(\la(t))=0$.   This point is called the {\em virtual center} of $\Omega_n$.   
\end{enumerate}
\end{thm}

Statement~\ref{dd} is Theorem C.
 
  Because of our conventions in defining the address of a pole, there are discontinuities in the address as $\la$ crosses the real axis.  By  propositions~\ref{posla} and~\ref{negescape}, the only component not completely contained in either the upper or lower half plane is the component $\Omega_1$ containing the real line. Therefore as  $\la$ moves in any other component, the prepoles are holomorphic functions of $\la$ and their addresses remain constant.

\begin{prop}\label{knseqsame} All functions in a given component $\Omega_n$ 
 share the same kneading sequence.
 \end{prop}
 
 \begin{proof}
Theorem~\ref{omeganprops} says that  pseudo-hyperbolic maps in a given component have the same dynamics.    Suppose first that $n=1$.  If $\la \in \Omega_1$, $f_{\la}$ has a unique fixed point with one completely invariant component and kneading sequence $S=*$. 

Suppose now that $n>1$.  The component $\Omega_n$ is completely contained either the upper or lower half plane.  
The  kneading sequence of the attracting cycle is defined by the index of the set $R_{k_1 \cdots k_n}$ that intersects the component $A_1$ of the cycle.   For all $\la \in \Omega_n$, the sets $A_1$ and $R_{k_1 \cdots k_n}$  intersect and share a common boundary point that is the unique prepole on the boundary of the component $A_1$ that is the preimage of $+\infty$ going backwards around the cycle.     Since the prepoles are holomorphic functions of $\la$,  the boundary of  $A_1$ moves holomorphically as does the set $R_{k_1 \cdots k_n}$. It follows that   the index of the set $R_{k_1 \cdots k_n}$ and hence the kneading sequence remains constant. 
\end{proof}

\begin{defn} The {\em bifurcation locus} $\calb=\calb(\mathcal{FP}_2)$ is the set of parameters that are not structurally stable.
\end{defn}

By the proposition above, $\calb$  is contained in the  complement of $\cup_n \mathcal H_n$.  See \cite{ABF} for a full characterization.  

\subsection{Characterization of shell components}
  It follows from   proposition~\ref{knseqsame} that  the  parameters in a given shell component $\Omega$ have the same kneading sequence.  This sequence determines whether they are   either all regular pseudo-hyperbolic or all  unipolar pseudo-hyperbolic or all  hybrid  pseudo-hyperbolic.    We therefore define
  \begin{defn}
  If $\Omega$ is a shell component containing regular parameters we call it a {\em regular shell component}, (or just a {\em regular component});  if $\Omega$ is a  shell component containing   unipolar parameters we call it a {\em unipolar shell component}, (or just  a {\em  unipolar component}) and if $\Omega$ is a  shell component containing hybrid   parameters we call it a {\em hybrid   shell component}, (or just  a {\em hybrid   component}).
  \end{defn}
  Figure~\ref{param} depicts the parameter plane.\footnote{The colors in parentheses refer to the color version of the paper.  They are graded shades of grey in the black and white version.}  The black and white regions correspond to round-off errors.  The large component on the right (yellow) is $\Omega_1$.   The unbounded components on the left are the unipolar components.   The largest (red) are period $3$  and the next largest, between the period $3$'s are period $5$'s (bright green).   The bounded bulb-like components tangent to $\Omega_1$ are regular components of varying periods: the largest (cyan) are period $2$, the next smaller ones are period $3$ red, and so on.

  Figure~\ref{param1} shows two different blow-ups of areas of figure~\ref{param}.  On the left  we have zoomed in near $0$ and we see period $3$ unipolar components at the top and bottom left.  The unipolar period $5$ components are more visible here, as are the period $7$'s (dark blue) between them and the period $9$'s (orange) between the $7$'s.  On the right    we have zoomed in on the period $2$ component with virtual center at $-\pi i$.   The largest hybrid components of period $4$  (khaki) that share this virtual center are now  visible and there are period $6$ (royal blue) components visible between them.   There are also regular period $4$ and $6$ components budding off of the period $2$ component.

\begin{figure}
  \centering
  \includegraphics[width=3in]{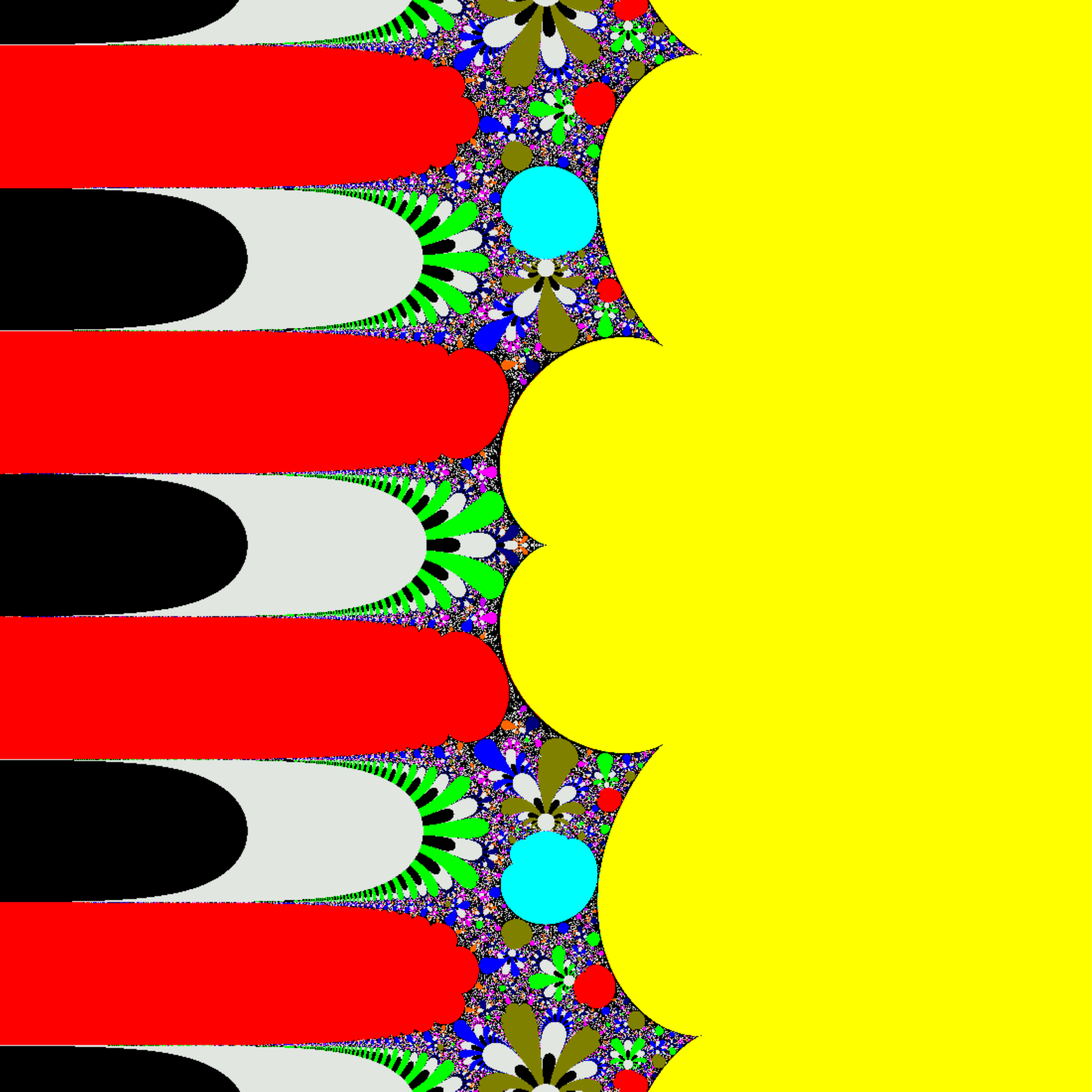}
  \caption{The parameter plane}
  \label{param}
\end{figure}

\begin{figure}[ht]
  \centering
  \begin{subfigure}[b]{0.5\linewidth}
    \centering\includegraphics[width=4.9cm]{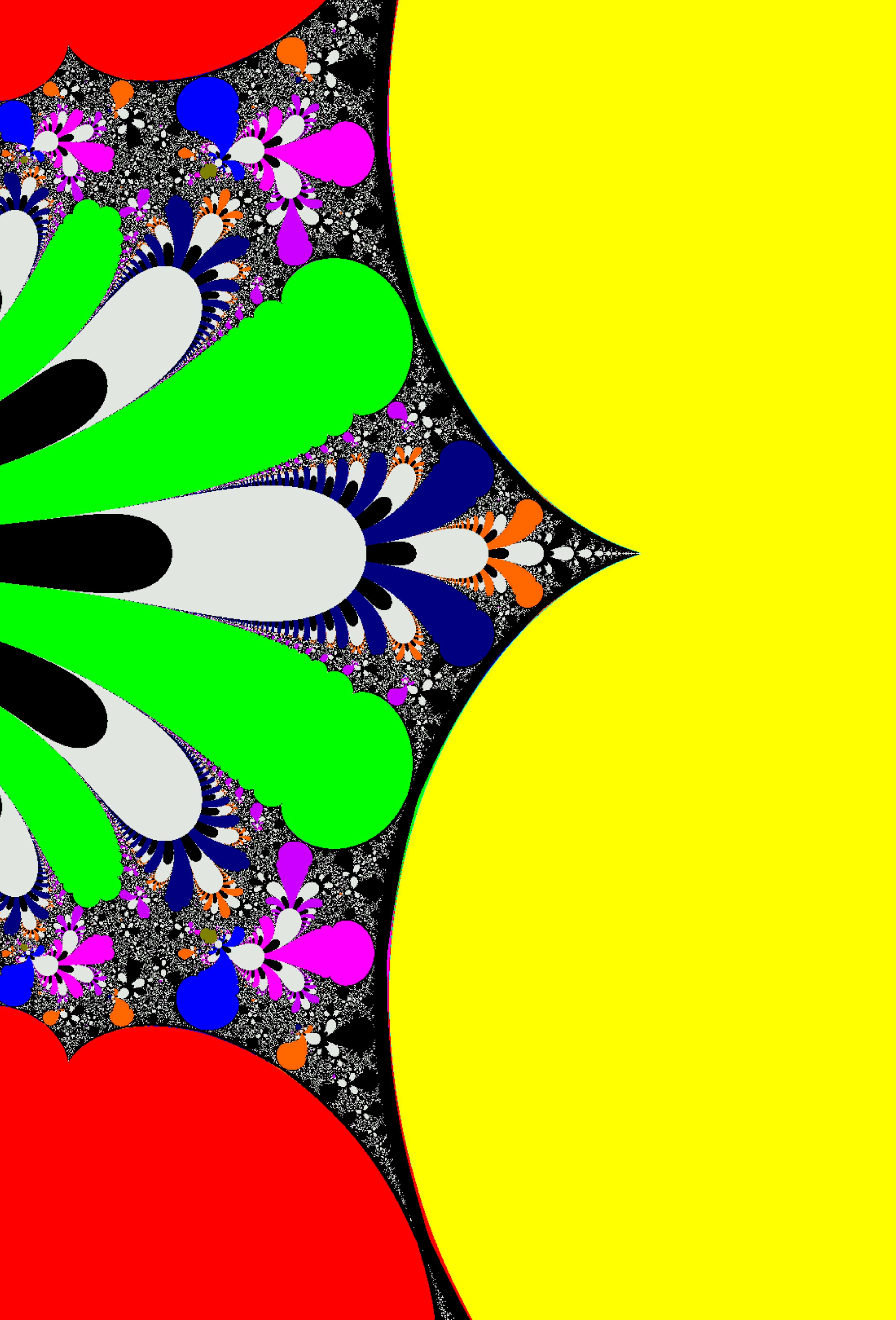}
    \caption{zoom in near $0$ \label{fig:fig1}}
  \end{subfigure}%
  \begin{subfigure}[b]{0.5\linewidth}
    \centering\includegraphics[width=6.5cm]{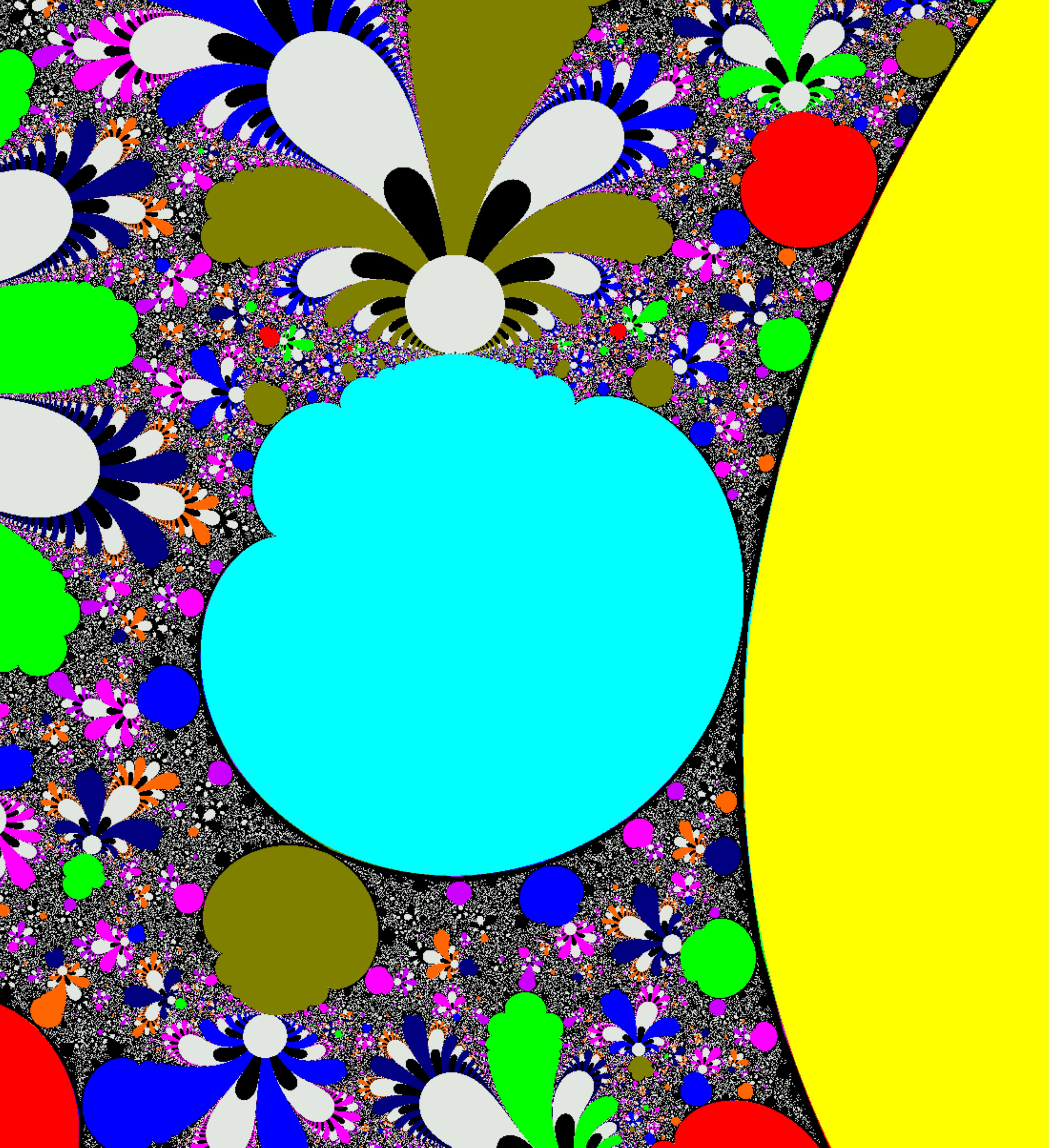}
    \caption{Zoom in near $-\pi i$\label{fig:fig2}}
  \end{subfigure}
  \caption{}
  \label{param1}
\end{figure}

The following proposition except for case (3c) was proved in \cite{FK} as proposition 6.8 and the proof goes through here without modification.
\begin{prop}\label{cycbehav}  Suppose $\Omega$ is a shell component of period $n\geq 1$ and let $\la_k$ be a sequence of points in   $\Omega$ such that $\lim_{k \to \infty} \la_k = \la^* \in \partial \Omega \cap \hat\CC^*$.  Let $\mathbf{z}^k=\{z_0^k, \ldots, z_{n-1}^k\}$ be the periodic cycle of $f_k=f_{\la_k}$ attracting $\la_k$.
\begin{enumerate}
\item If the cycle $\mathbf{z}^k$ stays bounded as $k \to \infty$, then $f_{\la^*}$ has an indifferent periodic cycle $\mathbf{z}^*$ whose period divides $n$ and $\mathbf{z}^k \to \mathbf{z}^*$.
\item If $n=1$, and $|z_0^k| \to \infty$ as $k \to \infty$, then $\Re z_0^k \to +\infty$, $\la_k \to +\infty$ and $\calm(\mathbf{z}^k) \to 0$.
\item  If $n >1$ and  $|z_j^k| \to \infty$ as $k \to \infty$ for some $0 \leq j  \leq n$,  then
\begin{enumerate}
\item If $\Omega$ is a regular component,   $f_{\la^*}^{j-1}(\la^*)$ is a non-zero pole and $j=n-1$.
 \item  If $\Omega$ is a hybrid component,  $j$ is the smallest index such that $f_{\la^*}^{j-1}(\la^*)= k_{l} \pi i$ where   $n=j+2m$ and $k_{l} \neq 0$  is  defined in  definition~\ref{knseqnew}; thus,  $k_l=k_{n-j}$.
 \item If $\Omega$ is a unipolar component, $\Re \la \rightarrow -\infty$,  $n=2m+1$, $j$ is odd and  the limit cycle is $\{+\infty, -\infty, 0, -\infty, \ldots, -\infty, 0\}$.
\end{enumerate}
\end{enumerate}
\end{prop}
In cases (3a) and (3b)   $\la^*$
 is a virtual cycle parameter as well as a virtual center.   Since $\la^*$ is a prepole for $f_{\la^*}$, it has an address.  In case (3a) its address is given by its kneading  sequence and in case (3b) its address is given by the first part of the sequence, $k_1 \ldots k_l$.  
In case (3c), the llimit $\la^*$ is  $-\infty$.  As $\la$ tends to $\infty$ in the unipolar component,  the immediate basin of $f_{\la}$ contains a cycle of paths with endpoints on the boundary.   All the odd components but the last   contain a path in the left half-plane tending to $-\infty$, all the even components  contain a path ending at $0$ and $A_0$ contains a path in the right half-plane whose endpoint is $+\infty$.

\begin{proof}[Proof of (3c)]  By definition,  if $\la$ is in a unipolar component it has kneading sequence $*k_10\cdots 0k_{n-2 }0$.   In the proof of lemma~\ref{separating} we saw that  $A_j^k$ is unbounded in the left half-plane for all odd $j$.  For all $m$, the sets $\{A_j^k + m \pi i \}$   are preimages of $f(A_j^k)$.  This periodicity implies  that  if $|z_j^k| \to \infty$, $\Re z_j^k \to -\infty$ and $z_{j+1}^k \to 0$. Moreover, $z_{j-1}^k$ tends to a pole.   If the pole is not $0$, $\la$ is not in a unipolar component so  the pole is zero and $|z_{j-2}^k| \to \infty$.  Continuing to pull back around the cycle we get   $\Re z_1^k \to -\infty$ and  $\Re z_0^k \to +\infty$.
\end{proof}

In the next section we prove that the converse holds:  we  construct shell components at each virtual cycle parameter and these are parameters are their virtual centers.

  \subsection{No virtual center at $0$}
 \label{zeronotvc}

By  theorem C, the virtual center of each   component in $\mathcal{H}_2$ is a virtual cycle parameter $\la^*$; since the period is $2$,  the   $\la^*$ is also a  pole  so $\la^*=\pi i k$ for some $k \in \ZZ$.  That is, it is a solution to $f_\lambda(\lambda)=\frac{\lambda}{1-e^{-2\lambda}}=\infty$.  
 Note that $\lim_{\lambda
 \to 0}\frac{\lambda}{1-e^{-2\lambda}}=1/2$, so that zero is not the virtual center of a period $2$ component.  Since $f_{\la}$ is not defined, zero is not a point of the pararmeter space. 
   In fact, we have more; we can now prove lemma~\ref{evenper} stated in section~\ref{p-h maps}.
 
 \begin{starlemma}{\rm{\bf ~\ref{evenper}}}
 The value $0$  cannot be the virtual center of a pseudo-hyperbolic component of any period. That is, there is no pseudo-hyperbolic $f_{\la}$ whose kneading sequence has either the form
    $*0$ or the form $*0k_20\cdots k_{n-2}0$.
        \end{starlemma}

\begin{proof} Given $f_\lambda$, let $w_n(\lambda)=f^n_{\la}(\lambda)$.  Because every virtual center is also  a virtual center parameter,  to prove this lemma  it suffices to show that $\lim_{\lambda\to 0} w_n(\lambda)\neq \infty$.  This follows from the following claim 
 which we prove by induction:\\
  For any $i\geq 0$, define the sequence   $a_i\in \mathbb{C^*}$  by
 $$a_0=\frac{1}{2}, \text{ and } a_{i+1}=\frac{1}{2(1-e^{-2a_i})}.$$
Then 
$$\lim_{\lambda\to 0} w_{2i}(\lambda)=0, \text{ and } \lim_{\lambda\to 0} w_{2i+1}(\lambda)=a_i.$$

Suppose $i=0$.  Then 
$$w_0(\lambda)=\lambda \mbox{ and } w_1(\lambda)=\frac{\lambda}{1-e^{-2\lambda}},$$
so  that $$\lim_{\lambda\to 0} w_{0}(\lambda)=0, \text{ and } \lim_{\lambda\to 0} w_{1}(\lambda)=1/2,$$  and the claim is true for $i=0$. 

Next, we assume that the claim is true for $i$, and  show that it is also true for $i+1$. 

By the induction hypothesis,  $w_{2i+1}(\lambda)\to a_i\in\mathbb{C}^*$,  so that   $$w_{2i+2}(\lambda)=\frac{\lambda}{1-e^{-2w_{2i+1}}}\simeq \frac{\lambda}{1-e^{-2a_i}}\to 0$$ as $\lambda\to 0$. 
Therefore as $\lambda\to 0$, $$w_{2i+3}(\lambda)=\frac{\lambda}{1-e^{-2w_{2i+2}}}\to \frac{1}{2(1-e^{-2a_i})}:=a_{i+1}\in \mathbb{C}^*,$$ 
completing the  proof of claim and hence also the proof of  the lemma.
\end{proof}

An immediate corollary is
\begin{cor}
The kneading sequence of every  pseudo-hyperbolic map is an  allowable sequence.
\end{cor}

In the next section  we will show that for any integer $k$, $k\neq 0$, there is a regular shell component whose  kneading sequence is $*k$.

\subsection{The set $\mathcal H_n$}
If $\Omega_n$ is a component of $\calh_n$,  proposition~\ref{knseqsame} says that the kneading sequence of every function in the component is the same.  Therefore, we may identify the kneading sequence of any one of these cycles with the component and set $S(\Omega)=*k_1\cdots k_{n-1}$.

The virtual center of $\Omega_n$ is assigned the same sequence.  Since we have infinitely many components meeting at a given virtual cycle parameter, this parameter corresponds to infinitely many sequences.  They all have a common initial segment consisting of non-zero entries and corresponding to the regular component whose virtual center is that virtual cycle parameter.

\section{Existence of pseudo-hyperbolic components}
\label{existence}

In this section we give a full description of how the   pseudo-hyperbolic components in the parameter plane are situated.   In particular, we prove that every virtual cycle parameter is the virtual center of one regular and infinitely many hybrid pseudo-hyperbolic components.  We also show that there are components corresponding to all allowable kneading sequences.

\begin{thm}[Theorem D]
\label{Main Theorem A}
 Given any allowable kneading sequence
$S=*k_1k_2 \ldots k_{n-1}$ with  $n>1$:
\begin{enumerate}

 \item If $S$ is a unipolar sequence,  $S=*k_10\cdots 0k_{n-2 }0$,  then $n$ is odd, and  there is a unipolar pseudo-hyperbolic component associated to it.  This component is unbounded in the left half-plane and, as $\Re{\la} \to -\infty$ in the component, the multiplier of the attracting cycle goes to $0$.
 \item If $S$ is a regular sequence, $S=*k_1\cdots k_{n-1}$, where $k_{n-1}\neq 0$, then
 there is a regular component associated to it, and its virtual center is the virtual cycle parameter with the same address.
 \item
 If $S$ is a hybrid sequence,  $S=*k_1\cdots k_{l-1}k_lk_{l+1}0\cdots k_{n-2}0$,  then the first part of sequence, $k_1\cdots k_{l},$ is the address of a regular virtual cycle parameter  and there is a hybrid component whose virtual center is this virtual cycle parameter and it  has kneading sequence $S$.
 \end{enumerate}
 \end{thm}

  Note that a regular sequence may   have an even period.   Since there are infinitely many kneading sequences that share a given initial segment, a  corollary of this theorem is
  \begin{cor}  There are infinitely many hybrid components that share a virtual center with one regular component.
  \end{cor}
\begin{remark}
    Although infinity is not a parameter, it is a boundary point of infinitely many  components for which it acts as a virtual center: the regular period $1$ component whose kneading sequence  is $*$   and the infinitely many pure unipolar components that are defined by sequences   that consist only of the pole followed by alternating $0$'s.
\end{remark}

We prove each part of the theorem separately.
  The proof of the first part takes advantage of the fact that,  restricted to neighborhoods of $0$ and $\infty$,  the second iterate of a map in $\calfp_2$ acts like an exponential map and the structure of the unipolar pseudo-hyperbolic components is reminiscent of  the structure of the hyperbolic components of exponential maps  described  in \cite{DFJ, S}.    The proof of the second part is very similar to  the proof of Theorem B in \cite{CK1}   in which the dynamics of a function whose asymptotic value is a prepole is perturbed. The proof of the third part combines techniques in the first two.

  \subsection{Proof of theorem~\ref{Main Theorem A} part $(1)$:   unipolar maps}\label{puremaps}

 It follows from the definition of the kneading sequence (and remark~\ref{lainR}) and from lemma~\ref{separating}, that to prove the existence of a $\la$ with the given kneading sequence $*k_10\cdots 0k_{n-2 }0$ it suffices  to show there is a $\la$ such that in the L-R structure of $f_{\la}$, $\la \in R_{k_1 0 \ldots  k_{n-2}0}(\la)$.  For such a $\la$,
 $$ R_{k_10\cdots 0k_{n-2 }0}(\la) \subset L_{k_10\cdots 0k_{n-4 }0} (\la)\subset L_{k_1 0 \ldots  0 k_{n-6}0} (\la)\cdots \subset L_{k_10}(\la) \subset L $$
   so that $f_{\la}^{n-1}(\la)$ is contained in a right half-plane; it is clear,  that $n-1$ must be  even.

 We will need the  basic estimates in ~(\ref{basic estimate}) for the exponential map which we will use to follow the orbit of $\la$.

 To construct the pseudo-hyperbolic components we proceed inductively.   The first step is to construct an $f_{\la}$ of period $3$ with $S=*k0$.  Since we have assumed the sequence is allowable, $k \neq 0$.   We first construct a ``candidate'' strip using the estimates above.  Then we prove that for $\la$ in that strip, $f_{\la}$ has an attracting cycle of period $3$.

We have seen that the parameter space is symmetric about the real axis and that there is an infinite segment of the positive real axis contained in a period $1$ component.   In this proof we will consider strips about the lines whose imaginary parts are odd multiples of $\pi$.   In order to keep the labeling consistent with the symmetry, and because $k \neq 0$, we will label these lines by $2k-1$ in the upper half-plane where $k>0$ and by $2k +1$ in the lower half-plane where $k<0$.

 \begin{lemma}\label{per3}
 Given the kneading sequence $S(\la)=*k0$, $k \neq 0$,
 there exists a large $M>0$ and a small $\epsilon>0$ such that
 if  $-2\la$ is in the half strip of the right half-plane centered along the line whose imaginary part is an odd multiple of $\pi$, that is, for $k>0$
  $$-2\lambda\in \calr_{k,M,\epsilon}^{odd}=\{ z= x+iy  \ | \ x>M, (2k-1) \pi-\frac{\pi}{2}+\epsilon <y< (2k-1)\pi+\frac{\pi}{2}-\epsilon\},$$
  and for $k<0$
 $$-2\lambda\in \calr_{k,M,\epsilon}^{odd}=\{ z= x+iy  \ | \ x>M, (2k+1) \pi-\frac{\pi}{2}+\epsilon <y< (2k+1)\pi+\frac{\pi}{2}-\epsilon\},$$
then  $f_{\la}$ has an attracting cycle of period $3$ with the given kneading sequence.
 \end{lemma}

 \begin{proof}

Set $w_0=\lambda$ and $w_n=f_\lambda^n(\lambda)$, $n=1, 2, \cdots$. Note that each $w_n$ is a holomorphic function of $\lambda$.  Using the estimates above, if for $M$ large enough, $-2\la \in \calr^{odd}_{k,M,\epsilon}$, then

 $$w_1=\frac{\lambda}{1-e^{-2\lambda}}\simeq -\frac{\lambda}{e^{-2\lambda}}\simeq 0$$  so that
$$w_2=\frac{\lambda}{1-e^{-2w_1}}\simeq \frac{\lambda}{2w_1}\simeq -\frac{e^{-2\lambda}}{2},$$
and  for small $\epsilon$, $\Re w_2  > e^M \sin \epsilon$.

 Choose and fix a $\la$ such that  $-2\la \in\calr^{odd}_{k,M,\epsilon}$   and consider its dynamical plane.   We will show $f_{\la}$ has an attracting period $3$ cycle.
   Let $U_2=\{z\  | \  \Re z>\Re w_2-1\}$, $U_1$ the component of $f^{-1}_{\lambda}(U_2)$ that contains $w_1$ and whose boundary contains $0$, and $U_0$  the component of $f^{-1}_\lambda(U_1)$ containing $\lambda$.   Because these components are distinct and only one may contain $\la$, $\la \not\in U_0 \cup U_1$; the pole zero is not inside any component.   Since $f^2_{\lambda}: U_0\to U_2$ is a  conformal homeomorphism, its inverse is well-defined;  denote it by $h_1$. By the Koebe $1/4$ theorem,
\begin{align*}
dist(\lambda, \partial U_0)\geq \frac{|h'_1(w_2)|}{4} &=\frac{1}{4}\frac{1}{|f'_{\lambda}(w_1)|}\frac{1}{|f'_{\lambda}(\lambda)|}\\
&=\frac{1}{4}\frac{|(1-e^{-2w_1})^2|}{|\lambda e^{-2w_1}|}\frac{|(1-e^{-2\lambda})^2|}{|\lambda e^{-2\lambda}|}\\
&\simeq \frac{1}{4}\frac{|4w_1^2|}{|\lambda|}\frac{|(1-e^{-2\lambda})^2|}{|\lambda e^{-2\lambda}|}=|\frac{w_1^2(1-e^{-2\lambda})^2}{\lambda^2e^{-2\lambda}}|=|e^{2\lambda}|\end{align*}
Furthermore, for any $z\in U_2$,

\begin{align*}|f_{\lambda}(z)-\lambda|=\frac{|\lambda e^{-2z}|}{|1-e^{-2z}|}< \alpha e^2|\lambda e^{-2 w_2}|\simeq \alpha e^2 |\lambda e^{-e^{-2\lambda}}| \end{align*} for some positive constant $\alpha<2$ depending on $M$ and $\epsilon$.

 Choosing $M$ and $\epsilon$ respectively large and small enough, we can assure that $\Re w_2>>0$ and $|e^{2\lambda}|>  \alpha e^2 |\lambda e^{-e^{-2\lambda}}|$ so that $f_{\lambda}^3(U_0)$ is properly contained in $ U_0$. Therefore  by the Schwarz lemma,  there is an attracting cycle in $U_0, U_1, U_2$ and the set  $R_{k0}$  of the L-R structure for $f_{\la}$ is contained in $U_0$.\footnote{Note that the notation  convention defined in section~\ref{p-h maps} for the labelling of the basins is reversed here so that  $w_i \in U_i$;  the basins of the cycle  are labelled $A_0, A_1, A_2$, with $\la \in A_1$, $U_0 \subset A_1, \, U_1 \subset A_2$ and $U_2 \subset A_0$.  We will continue this convention throughout this section.}
\end{proof}

It follows that for large enough $M$ and small enough $\epsilon$, there is a  unipolar pseudo-hyperbolic component of period $3$ with  kneading sequence $*k0$.  Moreover,
for each $\la$ such that $-2\la \in \calr^{odd}_{M,k,\epsilon}$, if $k>0$,  the   set  $R_{k0}(\la)$ of the L-R structure of the dynamical plane of $f_{\la}$   is contained in the left half strip
$$ \call_{M/2,\epsilon,k} =\{ z=x+iy  \, | \, x < -M/2, k \pi-\frac{3\pi}{4}+\epsilon <y< k\pi-\frac{\pi}{4}-\epsilon\}$$ and for $k<0$ it is contained in the left half strip
$$ \call_{M/2,\epsilon,k} =\{ z=x+iy  \, | \, x < -M/2, k \pi+\frac{\pi}{4}+\epsilon <y< k\pi+\frac{3\pi}{4}-\epsilon\}.$$

Since all the strips contain the half-lines, $z=x+i (2k \pm 1)\pi/2 $, where the sign is determined by the parity of $k$,  $\cap R_{k0}(\lambda)\neq \emptyset$ and this intersection is contained in the asymptotic tract of $0$ for each $f_{\la}$.   It follows that  we can identify this intersection with a neighborhood of infinity in the parameter plane $\CC$ and this neighborhood is contained in a period $3$ pseudo-hyperbolic component.

To prove  theorem~\ref{Main Theorem A} part $(1)$ for  all kneading sequences $S=*k_1 0 \ldots  k_{n-2}0$,  we generalize the construction above.
Because the period is odd, we  set $n = 2m+1$.  Then, using the kneading sequence,
we  create  pairs of disjoint strips in the left and right half  planes where again   $\la$ is in the strip in the left half-plane but now $w_{2m}=f^{2m}_{\la}(\la)$ is in the strip in the right half-plane.

More precisely, to find the attracting cycles of period $2m+1$, we need to look for $\lambda$'s, whose orbits $\{w_0, w_1, \cdots, w_{2m}\}$ satisfy

\begin{itemize}
\item $w_{2i-1}$, $i=1, \cdots, m$ is sufficiently close to zero,

\item $w_0,w_2, \cdots w_{2m-2} $  are in the left half-plane  and their real parts have sufficiently large absolute values.
\item $w_{2m}$ is in the right half-plane with sufficiently large real part.

\end{itemize}

To do this, we need to define a structure on the parameter plane that is analogous to the L-R structure we defined in the dynamical planes of each of the functions $f_{\la}$.
We  define  strips in the right half-plane centered on the lines whose imaginary parts, $2k\pi$, are even  multiples of $\pi$,
$$\mathcal{R}_{k, M, \epsilon}^{even}=\{ \lambda= x+iy\  | \ x>M, 2k\pi-\frac{\pi}{2}+\epsilon <y< 2k\pi+\frac{\pi}{2}-\epsilon\}.$$
  Assume $M>0$ is large and $\epsilon>0$ is small.  For readability set  $E(z)=e^z$.

The following lemma  is a restatement of the first half of  theorem~\ref{Main Theorem A} part $(1)$.
\begin{lemma}
\label{pureuniexist}
Given $n$  and an allowable unipolar kneading sequence $S=*k_10\cdots 0k_{n-2 }0$,  there is a $\la \in \CC^*$ such that $f_{\la}$ has an attracting cycle of  period $n$ with this kneading sequence.
In particular, there is a $\la$ such that for
 $i = 1, \cdots, m-2$,
$$E^i(-2\lambda)\in  \mathcal{R}^{even}_{k_{2i+1}, M,\epsilon}$$ and $$ E^{m-1}(-2\lambda)\in \mathcal{R}^{odd}_{k_{n-2},M,\epsilon}.$$
 \end{lemma}

\begin{proof}
 If $M$ is sufficiently large  and $\epsilon$ is sufficiently small,  by periodicity,  for any $k$, $I=E(\calr^{even}_{k ,M,\epsilon})$ is a region outside the disk $|z| < e^M$ and contained between the rays $ te^{\pm\epsilon i}$, $t > e^{M}$.   That is,  it is a  sector at infinity in the right half-plane, symmetric with respect to  the real line and lying outside the disk  $|z| < e^M$.   To insure that the disk intersects the line $\Re z = M$, we choose $\epsilon$ small enough so that $e^M \sin \epsilon >M$.  Note  that for all $k$ and $\Re z$  large enough,   all  the strips $\mathcal{R}^{even}_{k, M, \epsilon}$ and  $\mathcal{R}^{odd}_{k, M, \epsilon}$ intersect $I$.   Figure~\ref{strips}  illustrates this.

   Note that for any $k$, and for any $\la$ with  $-2\la \in \mathcal{R}_{k, M, \epsilon}^{even}$,     $E(-2\la) \in I$;  thus $w_2\simeq -E(-2\lambda)/2$ is in the left half-plane.  In particular, this is true for $k=k_1$.

     Now choose $\lambda$ so that   $-2\lambda\in \mathcal{R}^{even}_{k_1, M,\epsilon}$  and  also require that  its image  $E(-2\lambda)$ is in $I \cap \mathcal{R}^{even}_{k_2, M,\epsilon}$. Then,
      $$w_3=\frac{\lambda}{1-e^{-2w_2}}\simeq -\frac{\lambda}{e^{-2w_2}}\simeq -\frac{\lambda}{e^{e^{-2\lambda}}}\simeq 0$$ and $$w_4=\frac{\lambda}{1-e^{-2w_3}}\simeq \frac{\lambda}{2w_3}\simeq -\frac{e^{e^{-2\lambda}}}{2}=\frac{-E^2(-2\lambda)}{2}.$$

  As  above,  $E^2(-2\lambda)$ is contained in $I$ and
   $\Re E^2(-2\lambda)> e^{e^M\sin\epsilon}\sin \epsilon/2$;  thus $w_4\simeq -E^2(-2\lambda)/2$ is in the left half-plane and even further to the left than $w_2$.  See figure~\ref{strips} again.

 \begin{figure}
  \centering
  \includegraphics[width=3in]{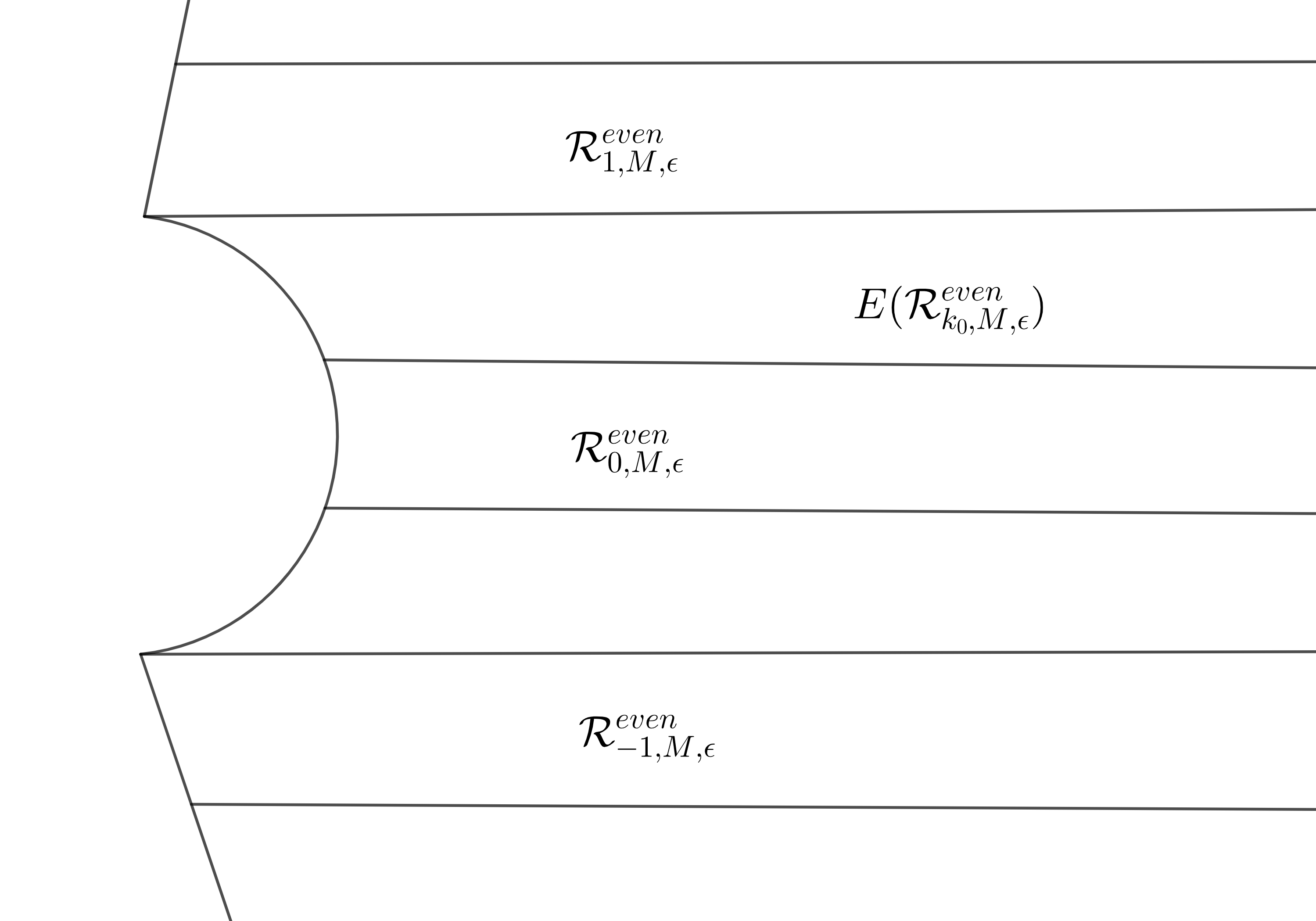}\\
  \caption{The intersection of the image $I$ of $E(\mathcal{R}^{even}_{k, M, \epsilon})$ for any $k$, with  the strips $I \cap  \mathcal{R}^{even}_{{k_i}, M,\epsilon}$}
 \label{strips}
\end{figure}

  We can iterate this procedure, choosing $\la$  so that for  $i= 1, \ldots, n-2$, the successive  images   $E^i(-2\la)$ lie in the strips $I \cap  \mathcal{R}^{even}_{{k_i}, M,\epsilon}$.   Having done this,   $w_{i}$ is  in the left half-plane,
     $$w_{2i-1}=\frac{\lambda}{1-e^{-2{w_{2i-2}}}}\simeq -\frac{\lambda}{e^{-2w_{2i-2}}}\simeq -\frac{\lambda}{E^{i}(-2\lambda)}\simeq 0$$ and $$w_{2i}=\frac{\lambda}{1-e^{-2w_{2i-1}}}\simeq \frac{\lambda}{2w_{2i-1}}\simeq \frac{-E^{i}(-2\lambda)}{2}.$$
  Now,  modify the choice of $\la$ so that, in addition to the above, it satisfies $E^{m-1}(-2\la)\in \mathcal{R}^{odd}_{k_{n-1}, M,\epsilon  }$.  Then, for such a $\la$,
     $w_{2i}\simeq -E^i(-2\lambda)/2$,  $i=0, \cdots, m-1$,  is in  the left half-plane,  and $w_{2m}\simeq - e^{E^{n-1}(-2\lambda)}/2$ is in the right half-plane.

  Let  $\tilde{V}$ denote the set of $\la$'s defined by the intersection of all these regions.  Any $\la \in \tilde{V}$ is a candidate to have an attractive cycle of period $n=2m+1$.   We now show it does.

Fix the candidate $\lambda$, and look at the dynamical plane of $f_\lambda$.  Let $U_{2m}=\{z\  | \  \Re z>\Re w_{2m}-1\}$, and let $U_{2m-i}$ be the component of $f^{-i}_{\lambda}(U_{2m})$ containing $w_i$, for $i=1, \cdots, 2m$. Then $0\in \partial U_{2i-1}$ and $U_{2i}$ is unbounded.  Since $f^{2m}_{\lambda}: U_0\to U_{2m}$ is conformal, its inverse is well-defined;    denote it by $h_n$. By the Koebe $1/4$ theorem,
\begin{align*}
dist(\lambda, \partial U_0) &\geq \frac{|h'_n(w_{2m})|}{4}  \\
&=\frac{1}{4}\frac{1}{|f'_{\lambda}(\lambda)|}\frac{1}{|f'_{\lambda}(w_1)|} \cdots \frac{1}{|f'_{\lambda}(w_{2m-2})|}\cdot \frac{1}{|f'_{\lambda}(w_{2m-1})|}\\
&=\frac{1}{4}\frac{|(1-e^{-2\lambda})^2|}{|\lambda e^{-2\lambda}|}\cdot \frac{|(1-e^{-2w_1})^2|}{|\lambda e^{-2w_1}|} \cdots \frac{|(1-e^{-2w_{2m-2}})^2|}{|\lambda e^{-2w_{2m-2}}|}\cdot \frac{|(1-e^{-2w_{2m-1}})^2|}{|\lambda e^{-2w_{2m-1}}|}\\
&\simeq \frac{1}{4}\frac{|(1-e^{-2\lambda})^2|}{|\lambda e^{-2\lambda}|}\frac{|4w_1^2|}{|\lambda|}\cdots \frac{|(1-e^{-2w_{2m-2}})^2|}{|\lambda e^{-2w_{n-2}}|}\frac{|4w_{2m-1}^2|}{|\lambda|}\\
&=\frac{1}{4} \Big|\frac{w_1^2(1-e^{-2\lambda})^2}{\lambda^2e^{-2\lambda}}\Big|\cdots \Big|\frac{w_{2m-1}^2(1-e^{-2w_{2m-2}})^2}{\lambda^2e^{-2w_{2m-2}}}\Big|\\
&=\frac{4^{m-1}}{|E(-2\lambda)\cdots E^{n}(-2\lambda)|}\end{align*}
Furthermore, for any $z\in U_{2m}$,

\begin{align*}|f_{\lambda}(z)-\lambda|=\frac{|\lambda e^{-2z}|}{|1-e^{-2z}|}< \alpha e^2|\lambda e^{-2 w_{2m}}|\simeq \alpha e^2 |\lambda E^{n+1}(-2\lambda)| \end{align*} for some constant $0<\alpha < 2$.

It is obvious now, that for our chosen  $M, \epsilon>0$,    $\la \in \tilde{V} $    satisfies the conditions:
   $$E^i(-2\lambda)\in \mathcal{R}^{even}_{k_i, M,\epsilon} \mbox{for  } i=0, \cdots, n-2,$$
   $$E^{n-1}(-2\lambda)\in \mathcal{R}^{odd}_{k_{n-1}, M,\epsilon},$$

   $$E^i(-2\lambda)\in \mathcal{R}^{even}_{k_{2i+1}, M,\epsilon} \mbox{for  } i=0, \cdots, m-2,$$
   $$E^{m-1}(-2\lambda)\in \mathcal{R}^{odd}_{k_{n-2}, M,\epsilon},$$
   $$\Re w_{2m}>>0 \mbox{ and } $$
$$\frac{4^{m-1}}{|E(-2\lambda)\cdots E^{m}(-2\lambda)|}> \alpha e^2 |\lambda E^{m+1}(-2\lambda)|.$$   Therefore $f_{\lambda}^{2m+1}(U_0)$ is properly contained in $ U_0$ so by the Schwarz lemma, there is an attracting cycle in $U_0, U_1, \cdots, U_{2m}$.

It is clear from the construction that the set $R_{0k_2 0  \ldots 0 k_{n-2} 0}(\la)$ of the L-R structure of $f_{\la}$ is a  subset of $U_0(\la)$ and that for these $\la$'s  the intersection of the $R_{0k_2 0  \ldots 0 k_{n-2} 0}(\la)$ is not empty.   This interesection is contained in the asymptotic tracts of $0$ for  each $f_{\la}$.  If we identify this intersection with a neighborhood of the parameter plane, we obtain an open set in a pseudo-hyperbolic component all of whose $\la$'s have the given kneading sequence.
\end{proof}

We have shown that there exists a unipolar parameter $\la=x+iy$ with an attracting cycle and given allowable kneading sequence.     Let $\Omega$ denote the pseudo-hyperbolic component containing $\la$. From the proof, it is clear that  any $\la'=x'+iy$, $x'<x$, is in $\Omega$.
\begin{prop}\label{multtozero}  Suppose $\{\la_j\} \in \Omega $, with  $\Re \la_j \to -\infty$ and let $m(\mathbf{z^j})$ be the multipler of the attracting cycle of period $2m+1$ of $f_{\la_j}$.   Then $\lim_{j \to \infty}  m(\mathbf{z^j}) =0$.
\end{prop}
\begin{proof}  We showed above that $f^{2m+1}$ contracts $U_0$ exponentially and that this contraction factor goes to zero with $\la$.  It follows that the multiplier must also go to zero.

 \end{proof}

Lemma~\ref{pureuniexist} together with this proposition prove part  $(1)$ of theorem~\ref{Main Theorem A}.

\begin{remark} It follows from construction in the proof above that components with kneading sequences $*00 \ldots k 0$ for fixed $k$ and increasing length converge to a segment in the negative real axis which we know
cannot be inside a hyperbolic component.

Note also that because the  parameter plane is symmetric about the real axis,  complex conjugation sends a component with kneading sequence $*k_1 \ldots k_{n-1}$ to one with kneading sequence $*-k_1 \ldots -k_{n-1}$.  This gives us another insight into lemma~\ref{separating}:   if there were a component with sequence $0 \ldots 0$, it would have to be  symmetric with respect to the real line and would therefore have to contain a s begment of it.  The segment cannot be in $\RR^+$ since that segment is in $\Omega_1$.   It would thus contain negative real values and, as we will see below,  these parameters are not pseudo-hyperbolic.  
\end{remark}

This theorem together with theorem~\ref{Omega1} prove that  infinity is like an ideal virtual cycle parameter.  It acts as the virtual center of    the  regular component $\Omega_1$ and of infinitely many  unipolar components of period $2n+1$, $n \in \NN$.

\subsection{Proof of theorem~\ref{Main Theorem A} part $(2)$: Regular  maps }

We restate theorem~\ref{Main Theorem A} part $(2)$ as

\begin{thm}\label{thmA} Let $\la_0$ be a virtual cycle parameter  whose address is $k_1 \ldots k_n$, $k_n \neq 0$. Then there is a regular pseudo-hyperbolic component with virtual center $\la_0$ whose attracting cycles have kneading sequence $S=*k_1 \ldots k_n$.
\end{thm}

The theorem follows from the more technical proposition below.

In the L-R structure of $f_{\la_0}$, let $U=R_{k_1 \ldots k_n}$ be the pullback of the asymptotic tract of $\la_0$ to $\la_0$.  Let $i: N \rightarrow \CC$ be the identity map from a neighborhood  $N$ of $\la_0$ in the dynamic plane of $f_{\la_0}$ to the parameter plane.  Let $R_M$ be a right half-plane contained in the asymptotic tract of $\la_0$.

Because $f_\lambda^n(\lambda)$ is holomorphic in $\la$ and $f_{\lambda_0}^n(\lambda_0)=\infty$, there exists a subset $\widetilde{V}\subset i(U)$ with $\lambda_0$ on the boundary such that,  for $\la \in \widetilde{V}$,  $f_\lambda^n(\lambda)$ is in a right half-plane $R_{M'}$ for some $M'>0$ close to $M$.
The  point is that  $\widetilde{V}$ is contained in a regular pseudo-hyperbolic component.

For any $\la$, define  $w_0=\la$, $w_i(\la)=f_{\la}^i(\la)$, $i=1, \ldots, n$, as points in the dynamic plane of $f_{\la}$.

\begin{prop}\label{regular} Let $\la_0$ be a virtual cycle parameter $\la_0$ whose address is $k_1 \ldots k_n$.  In the L-R structure of $f_{\la_0}$, let $R_{k_1 \ldots k_n}$ be the pullback of the asymptotic tract of $\la_0$ to $\la_0$ and identify it with a set $\widetilde{V}$ in the parameter plane.   Then
there exists a $V\subset \widetilde{V}$ such that for any $\lambda\in V$, the map $f_\lambda$ has an attracting cycle of period $n+1$.  Moreover, as $\la \to \la_0$ in $V$, the multiplier of the cycle tends to zero.
\end{prop}

The proof is similar to the proof of theorem~\ref{pureuniexist}  so we leave it to the reader to carry out.

\subsection{Proof of theorem~\ref{Main Theorem A} part $(3)$: Hybrid maps}
Finally we prove that in a neighborhood of a virtual cycle parameter $\la_0$ whose address is $k_1 \ldots k_l$,  $k_l \neq 0$,    for all $j \in \NN$  and all choices of $k_{l+1 }, \ldots k_{l+2j-1}$, with  $k_{l+1} \neq 0$, $k_{l+2}=0$, there are   hybrid   components with attracting cycles of period $n=l+2j+1$ whose kneading sequences are $S(\la)=*k_1 \dots k_l  k_{l+1} 0 \cdots  k_{l+2j-1}0$.

  As in part $(1)$ we have a virtual cycle parameter to use as  starting point in the parameter plane.   We look at the  L-R structure of $f_{\la_0}$ and  let $N$ be a neighborhood of $\la_0$ in the dynamic plane.  Now, however,  let $\widetilde{U}=L_{k_1 \ldots k_n}$ be the pullback of the asymptotic tract of $0$ to $\la_0$;  and let $i:N \rightarrow \CC$ be the identity map into the parameter plane;  set  $\widetilde{V}=i(\widetilde{U})$ in the parameter plane.  As above,   $w(\la)=f_{\la}(\la)$ is  holomorphic and the points
 $w_0=\la$ and  $w_i(\la)=f_{\la}^i(\la)$, $i=1, \ldots, n$ are  in the dynamic plane of $f_{\la}$.

Given $M>>0$, there exists a sector $V\subset \widetilde{V}$ at $\lambda_0$ such that for $\la \in V$,  $\Re w_n<-M$. Then

$$w_{l+1}=\frac{\lambda}{1-e^{-2w_{l}}}\simeq -\frac{\lambda}{e^{-2w_l}}\simeq 0$$  and
$$w_{l+2}=\frac{\lambda}{1-e^{-2w_{l+1}}}\simeq \frac{\lambda}{2w_{l+1}}\simeq -\frac{e^{-2w_{l}}}{2}.$$

Applying  these inequalities inductively as we did in the proofs of the two theorems above gives a  proof of   the existence part of part $(3)$ of theorem~\ref{Main Theorem A}.   That is,
\begin{thm}\label{6.7} For any $j \in \mathbb{N},$ there exists a sector $V\subset \widetilde{V}$  such that for any $\lambda\in V$,  $f_\lambda$ has an attracting cycle of period $n=l+2j+1$. Moreover, for any given sequence $k_1, \cdots, k_{l} \in \ZZ$,  for $i=l+1, \cdots, l+2j-1$,  $E^i(-2w_{n})\in \calr^{even}_{k_i, M,\epsilon}$,  and $E^{l+2j }(-2w_n)\in \calr^{odd}_{k_{j-1},M, \epsilon}$.
\end{thm}

It is clear from the construction that the virtual cycle parameter $\la_0$ is on the boundary of $V$. Let  $\Omega$ be the pseudo-hyperbolic component with kneading sequence $S$  constructed above.   To conclude the proof of theorem~\ref{Main Theorem A} we need to prove $\la_0$ is a virtual center.

\begin{prop}  Suppose $\{\la_m\} \in \Omega $, with  $  \la_m \to \la_0$ and let $\calm(\mathbf{z^m})$ be the multipler of the attracting cycle of of $f_{\la_m}$ of period $n+2j+1$.   Then $\lim_{m \to \infty}  \calm(\mathbf{z^m}) =0$.
\end{prop}
 This is  the analogue of proposition~\ref{multtozero} and the proof is the same.
\subsection{Escaping parameters}

Recall the definition of the zero cycle  $\{0,-\infty\}$ from  section~\ref{defzeroinfcycle}.
As we saw in proposition~\ref{negescape},  if  $\lambda <0$ it is attracted to this virtual cycle.  Such a $\la$,  however, cannot  be pseudo-hyperbolic and thus belongs to the bifurcation locus.   As we see below, there are other parameters attracted to the zero virtual cycle.

\begin{defn}\label{esc} If $f_{\la}^{2}(\la)$ approaches the zero virtual cycle in the sense that $f^{2n}_{\la}(\la) \to 0$ and $\Re f^{2n+1}_{\la}(\la) \to -\infty$ as $n \to \infty$ (or with odd and even iterates interchanged), we call $\la$ an {\em escaping parameter}.
\end{defn}

In  section~\ref{puremaps},   we found parameters with attractive periodic cycles that have a given kneading sequence.  In these cycles, the orbit of $\la$ alternated between the left half plane and a neighborhood of zero until, at the last step, it landed in the right half plane.    If $\la$ is escaping,  none of the points of its orbit ever get to the right half plane.   In the notation used to construct the cycles, we thus have 
  a proof of the following criterion for a parameter to be escaping.
  \begin{prop}\label{escpaths}
 If, given  $\lambda$ in the parameter plane,  there exists an $n$ such that $E^{i}(-2w_n) \in \mathcal{R}^{even}_{k, M, \epsilon}$ for all $i=0, 1, \cdots$, then $w_{n+2i+1}(\lambda)\to 0$ and $\Re w_{n+2i}\to -\infty$;  that is, $\la$ is escaping.
 \end{prop}

 Escaping exponential maps have been thoroughly investigated in \cite{FRS}; in particular, it was shown that  the set of escaping parameters consists of uncountably many disjoint curves in $\mathbb{C}$ tending to infinity, (see Theorem 1 in \cite{FRS}). The above proposition indicates that the set of escaping parameters of $\calfp_2$ also contains infinitely many such paths that extend  $\infty$ as well as sets of infinitely many disjoint curves tending to each virtual center parameter.  We can characterize these escaping parameters by assigning them infinite kneading sequences.  For example, by proposition~\ref{negescape}, the kneading sequence of parameters  in $\RR^-$ is $*0  \ldots 0 \ldots $.

\begin{prop}  If, given  $\lambda$ in the parameter plane, $S(\la)$ is an infinite kneading sequence $S=*k_1k_2 \ldots k_n 0 k_{n+2}0k_{n+4}0 \ldots k_{n+2j}0 \ldots$,     then  $\lambda$ is an escaping parameter.
\end{prop}

\begin{proof} From its kneading sequence we see that  $\la \in L_{k_1k_2 \ldots k_n 0 k_{n+2}0k_{n+4}0 \ldots k_{n+2j}0\ldots}$ for all $j$ so that $\la$ must be escaping.
\end{proof}

These propositions imply

\begin{thm}
If $\la$ is an escaping parameter, and if the entries in its infinite kneading sequence are bounded, then $f_\lambda$ is  not structurally stable.
\end{thm}

\begin{proof}
Suppose the kneading sequence corresponding to the escaping parameter $\la_0$ is $S=*k_1k_2 \ldots k_n 0 k_{n+2}0k_{n+4}0 \ldots k_{n+2j}0 \ldots$.  It follows as above that $\la_0 \in L_{k_1k_2 \ldots k_n 0 k_{n+2}0k_{n+4}0 \ldots k_{n+2j}0\ldots}$.

 As we did in the proof of theorem~\ref{Main Theorem A}, we can identify the sets of  R-L structure with  sets in the parameter space.  Then, in the parameter space, for each $j$ we can choose a $\la_j$ in the image of  $R_{k_1k_2 \ldots k_n 0 k_{n+2}0k_{n+4}0 \ldots k_{n+2j}0}$ so that $\la_j \to \la_0$.   By theorem~\ref{6.7}, for each $j$,   $E^i(-2w_{n}(\la_j))\in \calr^{even}_{k_i, M,\epsilon}$,  and $E^{j-1}(-2w_n(\la_j))\in \calr^{odd}_{k_{j-1},M, \epsilon}$. The boundedness condition implies that we can choose $\la_j$ so that $\Re \la_j$ does not go to $-\infty$ or to a virtual cycle parameter.   Then  each $f_{\la_j}$ is in a different pseudo-hyperbolic component.
\end{proof}

The structural instability of exponential maps $f_\lambda=\lambda e^z$ was first shown to exist for  $\lambda=1$ in \cite{D1}, then for $\lambda>e^{-1}$ in \cite{ZL}, and then at all escaping parameters at \cite{Y}. Our proposition~\ref{escpaths} indicates that maps are structurally unstable at uncountably many paths in the set of escaping parameters.

 We conclude the paper with a set of open questions:
 \begin{itemize}
 \item Are there other continua of escaping parameters in the bifurcation locus besides the negative real axis?
 \item The Julia sets of escaping parameters is always the whole sphere.  Does the set of points escaping to $\{0,\infty\}$ contain smooth curves?  Does this set have positive measure?
 \end{itemize}

\end{document}